\documentclass[a4paper,11pt]{article}
\usepackage[a4paper,margin=.95in,bmargin=1in]{geometry}
\usepackage{graphicx}
\usepackage{amssymb,amsmath, amsthm}
\usepackage{mathrsfs}
\usepackage{bbm,bbold}
\usepackage{changes}

\newcommand{\myboldmath}[1]{\mathbbm{#1}}
\newcommand{\R}{\myboldmath{R}}
\newcommand{\G}{\myboldmath{G}}

\newcommand{\N}{\myboldmath{N}}
\newcommand{\Z}{\myboldmath{Z}}

\newcommand{\1}{\mathbf{1}}
\newcommand{\U}{\mathcal{B}}

\theoremstyle{plain}
\newtheorem{theorem}{Theorem}
\newtheorem{proposition}[theorem]{Proposition}
\newtheorem{lemma}[theorem]{Lemma}
\newtheorem*{lemma*}{Auxiliary Lemma}

\theoremstyle{definition}

\newtheorem{remark}[theorem]{Remark}

\theoremstyle{remark}

\newcommand{\Lap}{\Delta}

\newcommand{\wlap}{\mathcal{L}}
\newcommand{\up}{\textup{up}}
\newcommand{\down}{\textup{down}}

\DeclareMathOperator{\lk}{lk}

\newcommand{\im}{\operatorname{im}}

\newcommand{\orcomp}[3][]{B^{{#2}}({#3})^\perp_{{#1}}}

\newcommand{\error}{E}

\begin{document}
\normalem

\title{On Eigenvalues of Random Complexes\thanks{An extended abstract of this paper appeared at SoCG 2012. Research supported by the Swiss National Science Foundation (SNF Projects 200021-125309 and 200020-138230).}}
\author{Anna Gundert, \footnote{Universit\"{a}t zu K\"{o}ln, Weyertal 86--90, 50923 K\"{o}ln, Germany. \texttt{anna.gundert@uni-koeln.de}. Work on this paper was conducted at the Institut f\"{u}r Theoretische Informatik, ETH Z\"{u}rich.} \and Uli Wagner\footnote{IST Austria, Am Campus 1, 3400 Klosterneuburg, Austria. {\texttt uli@ist.ac.at}.}}

\date{\today}

\maketitle

\begin{abstract}
We consider higher-dimensional generalizations of the normalized Laplacian and the adjacency matrix of graphs
and study their eigenvalues for the Linial--Meshulam model $X^k(n,p)$ of random $k$-dimensional simplicial complexes on $n$ vertices.
We show that for $p=\Omega(\log n/n)$, the eigenvalues of each of the matrices are a.a.s.~concentrated around two values. The main tool, which goes back to the work of Garland, are arguments that relate the eigenvalues of these matrices to those of graphs that arise as links of $(k-2)$-dimensional faces. Garland's result concerns the Laplacian; we develop an analogous result for the adjacency matrix.

The same arguments apply to other models of random complexes which allow for dependencies between the choices of $k$-dimensional
simplices. In the second part of the paper, we apply this to the question of possible higher-dimensional analogues of the discrete Cheeger inequality, which in the classical case of graphs relates the eigenvalues of a graph and its edge expansion. It is very natural to ask whether this generalizes to higher dimensions and, in particular, whether the eigenvalues of the higher-dimensional Laplacian capture the notion of \emph{coboundary expansion} --- a higher-dimensional generalization of edge expansion that arose in recent work of Linial and Meshulam and of Gromov; this question was raised, for instance, by Dotterrer and Kahle. We show that this most straightforward version of a higher-dimensional discrete Cheeger inequality fails, in quite a strong way: For every $k\geq 2$ and $n\in \N$, there is a $k$-dimensional complex $Y^k_n$ on $n$ vertices that has strong spectral expansion properties (all nontrivial eigenvalues of the normalised $k$-dimensional Laplacian lie in the interval $[1-O(1/\sqrt{n}),1+O(1/\sqrt{n})]$) but whose coboundary expansion is bounded from above by $O(\log n/n)$ and so tends to zero as $n\rightarrow \infty$; moreover, $Y^k_n$ can be taken to have vanishing integer homology in dimension less than $k$. 
\end{abstract}
%

\section{Introduction}
\label{sec:introduction}
Eigenvalues of graphs are a classical and well-studied subject, which goes back to a fundamental paper of Kirchhoff \cite{Kirchhoff:1847di}, in which he used the combinatorial graph Laplacian to analyze electrical networks and formulated his celebrated \emph{Matrix-Tree Theorem} for the number of spanning trees of a graph (which includes, as the special case of the complete graph, Cayley's \cite{Cayley:Trees-1889} famous formula $n^{n-2}$ for the number of labeled trees on $n$ vertices). 

The eigenvalues of a graph $G$ encode many important properties of $G$, in particular regarding connectivity and \emph{expansion} properties of $G$ (the mixing rate of a random walk on $G$) as well as other \emph{quasirandomness} properties of $G$. Because of this, eigenvalues of graphs also play a major role in the design and analysis of algorithms, including heuristic and approximation algorithms for hard graph partitioning problems (\emph{spectral partitioning}) and Markov Chain Monte Carlo approximation algorithms for hard counting problems. We cannot hope to survey the relevant literature here and refer the reader to the survey articles and monographs \cite{Chung:SpectralGraphTheory-1997,Jerrum:CountingSamplingIntegrating-2003,KrivelevichSudakov:PseudorandomGraphs-2006,
HooryLinialWigderson:ExpanderGraphs-2006,LevinPeresWilmer:MarkovChainsMixingTimes-2009,CvetkovicRowlinsonSimic:IntroductionGraphSpectra-2010,Zhang:2011un} for background and further references.

In the present paper, we consider eigenvalues of higher-dimensional simplicial complexes and, in a nutshell, prove two results: First, generalizing well-known results about random graphs $G(n,p)$, we show (Theorem~\ref{EigenvaluesRandomComplexes}) that the Linial--Meshulam $k$-dimensional random complexes are \emph{asymptotically almost surely} (\emph{a.a.s.}), i.e., with probability tending to $1$ as $n\rightarrow \infty$, strongly \emph{spectrally expanding} (their eigenvalues are strongly concentrated around two values). Second, we give a probabilistic construction (Theorem~\ref{thm:counterexample}) of $k$-dimensional complexes that are strong spectral expanders but that fail to have the property of \emph{coboundary expansion} --- a generalization of edge expansion that arose in the recent work of 
Linial and Meshulam \cite{LinialMeshulam:HomologicalConnectivityRandom2Complexes-2006} and of Gromov \cite{Gromov:SingularitiesExpanders2-2010}. This shows that the most straightforward attempt of generalizing the discrete Cheeger--Buser inequalities to higher-dimensional complexes fails and answers a question raised, e.g., by Dotterrer and Kahle \cite{Dotterrer:2010fk}.
Before stating these results more precisely, we first recall the basic definitions and terminology.

\subsubsection*{Adjacency Matrix and Laplacians of Graphs}
We recall the three ($n\times n$)-matrices commonly associated with a graph\footnote{Throughout this paper, we will assume that $G$ is simple, i.e., we do not consider loops or multiple edges.} $G=(V,E)$ on $n$ vertices. 
The \emph{adjacency matrix} $A=A(G)\in \{0,1\}^{V\times V}$ has entries defined by $A_{u,v}=1$ iff $\{u,v\}\in E$. 
The \emph{combinatorial Laplacian} is defined as $L=L(G):=D-A$, where $D=D(G)\in \R^{V\times V}$ is the diagonal matrix with entries $D_{v,v}=\deg_G(v)$, the \emph{degrees} of the vertices.
Both of these are \emph{symmetric} matrices and hence have a multiset of $n$ real eigenvalues, called the \emph{spectrum}.

The eigenvalues of $A$ and of $L$ turn out to be quite sensitive to the maximum and minimum degree of $G$. For graphs with very non-uniform degree distributions, it is often more convenient to consider the \emph{normalized Laplacian}, which is defined as $\Lap=\Lap(G):=D^{-1}L=I-D^{-1}A$, where $I\in \R^{V\times V}$ is the identity matrix.\footnote{%
Strictly speaking, $D^{-1}$ is defined only if there are no isolated vertices, i.e., if $\deg_G(v)>0$ for all $v\in V$, which will be the case of primary interest to us. If there are isolated vertices, we adopt the convention that $D^{-1}_{v,v}=0$ whenever $\deg_G(v)=0$ and retain the definition $\Lap=D^{-1}L$. (The second equation $\Lap=I-D^{-1}A$ no longer holds in this case, since $\Lap$ has zero diagonal entries at isolated vertices.)

Sometimes, (e.g., in \cite{Chung:2003wn,Chung:SpectralGraphTheory-1997,CojaOghlan:2007gj}) a slightly different matrix is referred to as the normalized Laplacian, namely $\mathscr{L}:=I-D^{-1/2}AD^{-1/2}$. Assuming that there are no isolated vertices, $\Lap$ and $\mathscr{L}$ have the same spectra, since $\Lap x=\lambda x$ for some $\lambda\in \R$ and $x\in \R^{V}$ iff $\mathscr{L}y=\lambda y$, where $y=D^{1/2}x$.} 

The normalized Laplacian is not symmetric but corresponds to a self-adjoint operator on $\R^n$ with respect to a weighted inner product (see  Section~\ref{sec:preliminaries}) and so also has $n$ real eigenvalues.
 Both versions of the Laplacian are \emph{positive semidefinite} relative to their respective inner products and so have nonnegative eigenvalues, typically listed in increasing order $\lambda_1(L)\!\leq\!\ldots\!\leq\!\lambda_n(L)$ and $\lambda_1(\Lap)\!\leq\!\ldots\!\leq\! \lambda_n(\Lap)$. The ``all-1'' vector ${\bf 1}=(1,\ldots,1)^T$ satisfies $L\1=\Lap\1=0$, hence $\lambda_1(L)=\lambda_1(\Lap)=0$, which is called the \emph{trivial eigenvalue}. 
For the adjacency matrix, the eigenvalues are typically listed in \emph{decreasing} order as $\mu_1(A)\!\ge\!\!\ldots\!\geq\! \mu_n(A)$. Define $\mu(G):=\max\{\mu_2(A),|\mu_n(A)|\}$.

The graph $G$ is connected iff  $\lambda_2(L)>0$ iff $\lambda_2(\Lap)>0$. More generally, the multiplicity of $0$ as an eigenvector of either Laplacian equals the number of connected components of $G$, and if $G$ is connected, then the second eigenvalue $\lambda_2$ of either Laplacian controls the \emph{edge expansion} of the graph (see the discussion below). 

\subsubsection*{Eigenvalues of Random Graphs} Let $G(n,p)$ be the binomial random graph on $n$ vertices, for which every edge is included independently with probability $p=p(n)$, and let $d=p(n-1)$ be the expected average degree. 
%
%
%
We summarize known concentration results on the spectra of $G(n,p)$ as follows.
See Section~\ref{subsec:EigenvaluesAdjacencyRandomGraphs} for a more detailed account.
\begin{theorem}[\cite{CojaOghlan:2007gj,Feige:2005hp,Hoffman:2012}] 
\label{thm:concentration-graph-eigenvalues}
For every $c>0$ and every $\gamma>c$ there exists a constant $C>0$ such that for $p\geq (1+\gamma)\cdot \log n/n$ and $d=p(n-1)$ the following statements hold with probability at least $1-n^{-c}$:
\begin{enumerate}
\item[\textup{(i)}] $\mu_1(A(G(n,p))) \in [d - C\cdot \sqrt{d}, d +C\cdot \sqrt{d}]$ and $\mu(G(n,p))\leq C\cdot \sqrt{d}$;
\item[\textup{(ii)}] $1-\frac{C}{\sqrt{d}}\!\leq\!\lambda_2(\Lap(G(n,p)))\!\leq\!\ldots\!\leq\!\lambda_n(\Lap(G(n,p)))\!\leq\!1+\frac{C}{\sqrt{d}}.$
\end{enumerate}
For the adjacency matrix \textup{(i)} even holds for $p\geq \gamma\cdot \log n/n$.
\end{theorem}

One type of application of such results 
is the analysis of spectral heuristics for algorithms that deal with random instances of NP-hard graph partitioning and related problems, see the discussions in \cite{Feige:2005hp,CojaOghlan:2007gj}.
%

\subsubsection*{Higher-Dimensional Laplacians} 
Eckmann~\cite{Eckmann:HarmonischeFunktionenRandwertaufgabenKomplex-1945} introduced a generalization of the graph Laplacian $L$ to higher-dimensional simplicial 
complexes $X$ to study discrete boundary value problems on such complexes. 

More precisely, let $X$ be a finite simplicial complex and let $C^i(X;\R)$, $i\in \Z$, be the vector space of $i$-dimensional simplicial cochains with real coefficients (we refer to Section~\ref{sec:preliminaries} for the necessary definitions). Eckmann defines three linear operators $L_i^\down(X)$, $L_i^\up(X)$ and $L_i(X)=L_i^\down(X)+L_i^\up(X)$ on the space $C^i(X;\R)$ and proves a discrete analogue of \emph{Hodge theory} \cite{Hodge:TheoryApplicationsHarmonicIntegrals-1989}, which implies, in particular, that the subspace $\mathcal{H}_i(X):=\ker L_i(X)$ of so-called \emph{harmonic cochains} on $X$ is isomorphic to $\widetilde{H}^i(X;\R)$, the $i$-th reduced cohomology.

In the case of a $1$-dimensional simplicial complex (graph) $G$, $L_0^\up(G)$ coincides with the usual graph Laplacian  $L(G)$ discussed previously.

Subsequently, combinatorial Laplacians were applied in a variety of contexts. Dodziuk~\cite{Dodziuk:FiniteDifferenceApproachHodgeTheoryHarmonicForms-1976} and Dodziuk and Patodi~\cite{DodziukPatodi:RiemannianStructuresTriangulationsManifolds-1976} showed how the continuous Laplacian of a Riemannian manifold can be approximated by the combinatorial Laplacians of a suitable sequence of successively finer triangulations of the manifold. 

Kalai \cite{Kalai} used combinatorial Laplacians to prove a higher-dimensional generalization of Cayley's formula for the number of labeled trees, and further results in this direction, including a generalization of the Matrix-Tree Theorem, were obtained in \cite{Adin:1992dx,Duval:2009jo}. For further combinatorial applications, see, e.g., \cite{FriedmanHanlon:BettiNumbersChessboardComplexes-1998,Friedman:ComputingBettiNumbersViaCombinatorialLaplacians-1998,MR1697094,MR1912799}. For 
further background and references regarding combinatorial Laplacians, see also~\cite{HorakJost}.

We will mostly work with a normalized version of the Laplacian, $\Lap_i(X)=\Lap_i^\down(X)+\Lap_i^\up(X)$ (see Section~\ref{sec:preliminaries} for the definition) and focus on the operator $\Lap_{k-1}^\up(X)$.
Again, for graphs, $\Lap_0^\up(G)$ agrees with the normalized graph Laplacian $\Lap(G)$ discussed above.

\subsubsection*{Random Complexes}
Linial and Meshulam~\cite{LinialMeshulam:HomologicalConnectivityRandom2Complexes-2006} introduced a higher-dimensional analogue of the binomial 
random graph model $G(n,p)$. By definition, the random $k$-dimensional complex $X^k(n,p)$ has $n$ vertices, a \emph{complete $(k-1)$-skeleton} (i.e., every subset of $k$ of fewer vertices form a face of the complex), and every $(k+1)$-element set of vertices is taken as a $k$-face independently with probability $p$, which may be constant or, more generally, a function $p(n)$ depending on $n$. 

This model has been studied extensively, and \emph{threshold probabilities} for several basic topological properties of $X^k(n,p)$ have been determined quite precisely, see e.g.~\cite{MeshulamWallach:HomologicalConnectivityRandomComplexes-2009,BabsonHoffmanKahle:SimpleConnectivityRandom2Complexes, Aronshtam:CollapsibilityVanishingTopHomologyRandomComplexes-2013,CohenCostaFarberKappeler:2012,Kozlov:2009p2037, Wagner:MinorsRandomExpandingHypergraphs-2011}. 

Our first result is a higher-dimensional analogue of Theorem~\ref{thm:concentration-graph-eigenvalues}.
The adjacency matrix of a $k$-dimensional complex $X$ is denoted by $A_{k-1}$ (see Section~\ref{sec:matrices-complexes} for the precise definition).
Both $A_{k-1}$ and the normalized up-Laplacian $\Lap_{k-1}^\up$ have rows and columns indexed by the $(k-1)$-faces of $X$; we assume that $X$ has $n$ vertices and a complete $(k-1)$-skeleton, so the matrices have dimension $\binom{n}{k}\times\binom{n}{k}$.
$A_{k-1}$ has entries in $\{0,\pm 1\}$, and $(A_{k-1})_{F,G}=\pm 1$ (with appropriate signs) iff $F\cup G$ is a $k$-face of $X$.

\begin{theorem}\label{EigenvaluesRandomComplexes}\xdef\savedtheoremnumber{\thetheorem}
Let $k\geq2$. For every $c>0$ and every $\gamma > c$ there exists a constant $C>0$ with the following property:
Assume  $p \geq (k+\gamma)\log(n)/n$ and let\footnote{Thus, $d$ is the expected \emph{degree} of any $(k-1)$-face $F$ in $X^k(n,p)$, i.e., the expected number of $k$-faces incident to $F$.} $d:= p(n-k)$. 
Then for $\gamma_A=C\cdot\sqrt{d}$ and $\gamma_\Delta=C/\sqrt{d}$ the following statements hold with probability at least $1-n^{-c}$:

\begin{enumerate}
\item[\textup{(i)}] The largest $\binom{n-1}{k-1}$ eigenvalues of $A_{k-1}(X^k(n,p))$ lie in the interval $[d-\gamma_A,d+\gamma_A]$, and the remaining $\binom{n-1}{k}$ eigenvalues lie in the interval $[-\gamma_A,+\gamma_A]$.
\item[\textup{(ii)}] The smallest $\binom{n-1}{k-1}$ eigenvalues of $\Lap_{k-1}^\up(X^k(n,p))$ are \textup{(}trivially\textup{)} zero, and the remaining $\binom{n-1}{k}$ eigenvalues lie in the interval $[1-\gamma_\Delta,1+\gamma_\Delta]$. In particular, $\tilde{H}^{k-1}(X^k(n,p);\R)=0$.
\end{enumerate} 
For the adjacency matrix \textup{(i)} even holds for $p\geq \gamma\cdot \log n/n$.
\end{theorem}
%

Both concentration results are achieved by reducing the higher-dimensional problem to estimates for the eigenvalues of random graphs, i.e., to Theorem~\ref{thm:concentration-graph-eigenvalues}. For the Normalized Laplacian this is done by applying a fundamental estimate due to Garland~\cite{Garland} 
that relates the eigenvalues of the higher-dimensional matrix to those of the graphs that arise as links of $(k-2)$-dimensional faces. For the generalized adjacency matrix we develop an analogous result (see Section~\ref{sec:Garland}).

Compared to the extended abstract \cite{GundertWagner-2012} of this paper, Theorem~\ref{EigenvaluesRandomComplexes} contains an improved concentration for the eigenvalues of $A_{k-1}$ in intervals of width $O(\sqrt{d})$ around the typical eigenvalues, as opposed to $O(\sqrt{d\log n})$.

Theorem~\ref{EigenvaluesRandomComplexes} also applies to any other random model for simplicial complexes with $n$ vertices and complete $(k-1)$-skeleton in which the links of $(k-2)$-faces are random graphs with distribution $G(n-k+1,p)$. 
We use this for our second result, a probabilistic construction of a counterexample for a conjectural higher-dimensional discrete Cheeger inequality (Theorem~\ref{thm:counterexample} below).

\subsubsection*{Edge Expansion and the Cheeger Inequality for Graphs}
For a graph of arbitrary density, its \emph{edge expansion} can be defined as follows. Let $\varepsilon>0$ be a parameter. We say that $G=(V,E)$ is $\varepsilon$-edge expanding if for every $S\subseteq V$,
\begin{equation}
\label{eq:edge-expansion}
\frac{|E(S,V\setminus S)|}{|E|} \geq \varepsilon \cdot \frac{\min\{|S|,|V\setminus S|\}}{|V|},
\end{equation}
where $E(S,V\setminus S)=\{\{u,v\}\in E: u\in S,v\in V\setminus S\}$ is the set of edges across the cut $(S,V\setminus S)$. Moreover, we call the best possible constant $\varepsilon$ the \emph{edge expansion}  of $G$ and denote it by $\varepsilon(G)$.\footnote{Note that (\ref{eq:edge-expansion}) is equivalent, to the more common condition that $|E(S,V\setminus S)| \geq\frac{\varepsilon}{2} \cdot d\cdot |S|$ for all $S\subseteq V$ with $|S|\leq |V|/2$, where $d=2|E|/|V|$ is the average degree. Thus, $\varepsilon(G)=2h(G)$, where $h(G):=\min\{\frac{|E(S,V\setminus S)|}{d|S|}:S\subseteq V,|S|\leq |V|/2\}$ is the (normalized) \emph{Cheeger constant} 
of $G$.} 
For a survey of the numerous applications of graph expansion in theoretical computer science and connections to other branches of mathematics, we refer to \cite{HooryLinialWigderson:ExpanderGraphs-2006}. 

As mentioned above, the edge expansion of a graph is controlled by the second-smallest eigenvalue of its Laplacian. Here, we state this fact in its simplest form, for $d$-regular graphs (due to Dodziuk \cite{Dodziuk:DifferenceEquations-1984}, Alon and Milman \cite{Alon:1985jg,Alon:1986wi}; Cheeger~\cite{Cheeger:LowerBoundSmallestEigenvalueLaplacian-1970} proved an analogous result for Laplacians on Riemannian manifolds.). A version for non-regular graphs, with a slightly different notion of edge expansion, can be found, e.g., in \cite{Chung:SpectralGraphTheory-1997}.
\begin{theorem}[Discrete Cheeger Inequality]\label{cheeger}\hspace{0.05cm}
Let $G=(V,E)$ be a $d$-regular graph, and let $\lambda_2=\lambda_2(\Lap(G))$ be the second-smallest eigenvalue of its normalized Laplacian. 
Then the edge expansion $\varepsilon(G)$ satisfies
$$\lambda_2 \leq \varepsilon(G) \leq \sqrt{8\lambda_2}.$$
\end{theorem}

The inequality on the left-hand side is proved fairly easily by expressing the characteristic function $\1_S\in \R^V$ of a subset $S\subseteq V$ as a linear combination of eigenvectors of the Laplacian $\Lap$. We will refer to this as ``\emph{the easy part of the Cheeger inequality}.''
The harder part is the inequality on the right-hand side. For a short proof see, e.g., \cite{AlonSchwartzShapira:ElementaryConstructionExpanders-2008}.

We remark that even the easy part of the Cheeger inequality is very useful. For instance, essentially all explicit constructions of constant-degree expanders \cite{Margulis:ExplicitConstructionsExpanders-1973,GabberGalil:ExplicitConstructionsSuperconcentrators-1981,LubotzkyPhillipsSarnak:RamanujanGraphs-1988,Margulis:ExplicitGroupTheoreticConstructions-1988,ReingoldVadhanWigderson:ZigZagProduct-2002}
%
prove a lower bound on the edge expansion of the constructed graphs by analyzing their eigenvalues.

\subsubsection*{Higher-Dimensional Expansion
} Recently, a higher-dimensional analogue of edge-expansion of graphs, \emph{coboundary expansion} (more precisely, $\Z_2$-coboundary expansion), arose in the recent work of Gromov~\cite{Gromov:SingularitiesExpanders2-2010} and of Linial, Meshulam and Wallach~\cite{LinialMeshulam:HomologicalConnectivityRandom2Complexes-2006,MeshulamWallach:HomologicalConnectivityRandomComplexes-2009}. 
The precise definition will be given in Section~\ref{sec:preliminaries}. (For further related results, see, also \cite{Fox:2010uq,Karasev:Gromov-2010,Newman:2011,MatousekWagner:AlsoSprachGromov-2011,Dotterrer:2010fk}.) 

It is natural to ask whether there is a higher-dimensional analogue of the discrete Cheeger inequality; this question was raised explicitly, e.g., by Dotterrer and Kahle~\cite{Dotterrer:2010fk}. As our second result we show, by a simple probabilistic construction, that the most straightforward attempt at a higher-dimensional Cheeger inequality fails, even for the ``easy part''. In higher dimensions, \emph{spectral expansion} (an eigenvalue gap for the Laplacian) does not imply $\Z_2$-coboundary expansion:

\begin{theorem}\label{thm:counterexample}
For every $k>1$ there is an infinite family of $k$-dimensional complexes $(Y^k_n)_{n \in \N}$, where $Y^k_n$ has $n$ vertices, that is \emph{spectrally but not coboundary expanding} in dimension $k$.

More precisely, all nontrivial eigenvalues of $\Lap_{k-1}^\up(Y^k_n)$ are $1\pm O(1/\sqrt{n})$,
but every $Y_n$ contains a cochain $a \in C^{k-1}(Y_n;\Z_2)$ of normalized Hamming weight $\|[a]\| \geq \frac{1}{2}-o(1)$ with $\|\delta a\|=O(\log n/n)$. Furthermore, $Y_n$ can be chosen such that $H_i(Y_n;\Z) = 0$ for all $i\leq k-1$.
\end{theorem}


For a graph $G$ and any abelian group $\G$, $\tilde{H}^{0}(G;\G)=0$ iff $G$ is connected. In higher dimensions, however, it is well-known that the vanishing of a cohomology group may depend on the choice of coefficients. A basic example for this is the real projective plane $\R P^2$ for which $\tilde{H}^1(\R P^2;\R) = 0$ but $\tilde{H}^1(\R P^2;\Z_2) = \Z_2$. In general, $\tilde{H}^1(Y;\G) = 0$ iff $Y$ is $\varepsilon$-expanding, with respect to a given norm on $\G$-cochains, for some small $\varepsilon>0$ that may depend on $Y$. Thus, the point of Theorem~\ref{thm:counterexample} is that there is an infinite family of examples whose coboundary expansion tends to zero (as fast as $\log n/n$) while the spectral expansion is bounded away from zero (in fact, equal to $1\pm O(1/\sqrt{n})$).

Compared to the extended abstract \cite{GundertWagner-2012} of this paper, the probabilistic construction behind Theorem~\ref{thm:counterexample} has been adapted to also allow for $H_{k-1}(Y_n;\Z)$ to be trivial. To influence the random behaviour we choose two probabilities $p,q \geq C\cdot \log(n)/n$ for suitably large $C$ with $q=o(p)$.
The construction then covers a whole range of parameters:
\[
|f_k(Y_n)-\tfrac{p}{2} \tbinom{n}{k+1}| \leq o(1) \tfrac{p}{2}\tbinom{n}{k+1},\quad \|\delta a\|=O\Big(\frac{q}{p}\Big),
\] while all nontrivial eigenvalues $\Lap_{k-1}^\up(Y_n)$ lie in the interval $[1-\gamma,1+\gamma]$ with $\gamma=O\big(1/\sqrt{(p/2)n}\big)$.

The concentration of eigenvalues is essentially optimal, as one can show\footnote{This can be shown analogously to the corresponding bound \eqref{BoundAdjacencySpectrumGraph} for graphs, see Preliminaries.} that  $\Lap_{k-1}^\up(X)$ always has a non-trivial eigenvalue $\lambda$ with $1-\lambda \geq \sqrt{k/d_{\max}\cdot(n-d_{\max})/(n-k)}$, where $d_{\max}$ is the maximal degree of a $k$-face in $X$, and the expected degree in $Y_n$ is $O((p/2)n)$.

In the extremal case $q = C\cdot \log(n)/n$ and $p=1$, we achieve a coboundary expansion of order $O(\log(n)/n)$ and eigenvalue concentration in $[1-O(1/\sqrt{n}),1+O(1/\sqrt{n})]$.

Of course it is just as natural to ask whether the other (``non-easy'') part of the Cheeger inequality has a simple higher-dimensional generalization. Even though any simplicial complex with non-zero $\Z_2$-coboundary expansion has to have non-zero spectral expansion, it has been shown that also for this part of the Cheeger inequality no straight-forward generalization can hold in higher dimensions: There is an infinite family of simplicial $k$-balls $X_n$ with spectral expansion $O(1/\log(n)^{\log(k)})$ and coboundary expansion $\Omega(1/\log(n))$, see~\cite{Steenbergen:2012}.
To the best of our knowledge, it is an open question whether there are complexes
with coboundary expansion bounded away from zero and spectral expansion tending to zero.

\subsubsection*{Related Work}
A recent article by Steenbergen, Klivans and Mukherjee~\cite{Steenbergen:2012} also presents a class of counterexamples for the most straightforward attempt at a higher-dimensional Cheeger inequality -- an explicit construction for an infinite family of simplicial $k$-balls $X_n$ whose spectral expansion is bounded away from zero, while the coboundary expansion tends to zero.
Here, the non-trivial eigenvalues of $\Lap_{k-1}^\up(X_n)$ are bounded below by a constant depending on the dimension $k$, while the coboundary expansion of $X_n$ is of order $1/\Theta(\log(n))$.
In the same article, the authors present the counterexample for simple higher-dimensional generalizations of the other (``non-easy'') part of the Cheeger inequality mentioned above.

Chung~\cite{Chung:1993} studies a higher Laplacian for hypergraphs that is closely related\footnote{One difference is that Chung's Laplacian operates not just on cochains, i.e., skew-symmetric functions on oriented simplices, but on arbitrary real-valued functions.} to the combinatorial Laplacian $L_{k-1}=L_{k-1}^\up+L_{k-1}^\down$. In \cite[Section~7]{Chung:1993}, she proves a somewhat weaker concentration result
for eigenvalues of random hypergraphs, namely, essentially, that for \emph{constant} $p$ and any $\varepsilon>0$, the eigenvalues of $L_{k-1}(X^k(n,p))$ are concentrated in an interval of width $O(n^{1/2+\varepsilon})$. She also states, without proof, that the proof methods for random graphs can be extended to yield the sharp bound of $O(\sqrt{pn})$. 

The probabilistic construction of the examples in Theorem~\ref{thm:counterexample} is well-known in the study of quasirandomness for hypergraphs, see, e.g., the discussion in \cite[Section~5]{Gowers:3UniformHypergraphs-2006}. In  \cite[Section~8]{Chung:1993}, it is asserted, again without proof,
that the eigenvalues of the combinatorial Laplacian of these examples are concentrated in an interval of width $O(\sqrt{pn})$, but we are not aware of a proof appearing in the literature.

Hoffman, Kahle and Paquette prove closely related results in their preprint~\cite{Hoffman:2012}.
They improve previous results on eigenvalues of random graphs and achieve precise information about the constant factor in the threshold.
Using a result by \.{Z}uk \cite{Zuk:1996}, which is a strengthening of Garland's estimate, they obtain as an immediate corollary that for $p\geq (2+\varepsilon)\frac{\log n}{n}$, the fundamental group of the random $2$-complex $X^2(n,p)$ a.a.s. has Property~(T). 

Using a weaker combinatorial notion of higher-dimensional expansion, but the same notion of Laplacian spectra, Parzanchevski, Rosenthal and Tessler show a version of a higher-dimensional Cheeger inequality~\cite{Parzanchevski:2012}. While $\Z_2$-coboundary expanding complexes also possess this weaker notion of expansion, the converse is not true (see, e.g., \cite{GundertSzedlak-2014}, where an extension of their result is presented).

In another recent article, Lu and Peng~\cite{Lu:2011} study a rather different kind of Laplacian for random complexes. Specifically, given a $k$-dimensional complex $X$ on a vertex set $V$ and a parameter $s\leq \frac{k+1}{2}$, they consider an auxiliary weighted graph on the vertex set $\binom{V}{s}$ in which $I,J\in \binom{V}{s}$ are connected by an edge of weight $w$ if $I\cap J=\emptyset$ and $I$ and $J$ are contained in precisely $w$ common $k$-faces of $X$. Lu and Peng study the normalized Laplacian of this auxiliary weighted graph. However, this Laplacian seems to capture the topology of $X$ only in a limited way. For instance, in the case $k=2$ and $s=1$, any two $2$-dimensional complexes on $n$ vertices that have a complete $1$-skeleton and are $d$-regular (every edge is contained in $d$ triangles) yield the same auxiliary graph, even though the topologies of these complexes (as measured by real cohomology groups and the usual Laplacian, say) may be very different.

\section{Preliminaries}
\label{sec:preliminaries}
\subsection{More on Eigenvalues of Graphs}

It is known that the spectrum of the normalized Laplacian $\Lap$ is contained in the interval $[0,2]$, and that $\lambda_n(\Lap)=2$ iff $G$ has a nontrivial \emph{bipartite} connected component \cite[Lemma 1.7]{Chung:SpectralGraphTheory-1997}. Moreover, if $G$ has no isolated vertices then 
$\lambda_{n-1}(\Lap)\geq \frac{n}{n-1}$.

If $G$ is \emph{$d$-regular}, i.e., $\deg_G(v)=d$ for all $v\in V$ (where $d$ may depend on $n$), then $L=d\cdot I-A=d\cdot \Lap$, and so the spectra of $A$, $L$, and $\Lap$ are equivalent (up to scaling and linear shifts): $\lambda_i(L)=d\cdot \lambda_i(\Lap)$ and $\mu_i(A)=d-\lambda_i(L)$, $1\leq i\leq n$. In particular, $\mu_1(A)=d$, $\mu_2(A)<d$ iff $G$ is connected, and $\mu_n(A)=-d$ iff $G$ has a nontrivial bipartite connected component. 

For $\mu(G)=\max\{\mu_2(A),|\mu_n(A)|\}$, it is not hard to show that for every $d$-regular graph
\begin{equation}\label{BoundAdjacencySpectrumGraph}
\mu(G) \geq \sqrt{d\cdot(n-d)/(n-1)}
\end{equation}
(see, e.g., \cite[Claim~2.8]{HooryLinialWigderson:ExpanderGraphs-2006}). Hence $\mu(G) \geq \Omega(\sqrt{d})$ for $d\leq 0.99n$, say, which shows that the concentration results for the eigenvalues of random graphs are essentially optimal.
For constant $d$, one has the sharper \emph{Alon-Boppana bound} $\mu(G) \geq  2\sqrt{d-1}\cdot (1-O(1/\log^2 n))$, see \cite{Nilli:SecondEigenvalueGraph-1991,Friedman:GeometricAspectsGraphsEigenfunctions-1993}.

A $d$-regular graph $G$ is called a \emph{Ramanujan graph} if it meets this bound for the spectral gap, i.e., if $\mu(G)\leq 2\sqrt{d-1}$.
It is a deep result due to Lubotzky, Phillips and Sarnak~\cite{LubotzkyPhillipsSarnak:RamanujanGraphs-1988} and independently to Margulis~\cite{Margulis:ExplicitGroupTheoreticConstructions-1988} that for every fixed number $d$ with $d-1$ prime, there exist Ramanujan graphs on $n$ vertices for infinitely many $n$ (and moreover, these graphs can be explicitly constructed). Recently, the existence of bipartite Ramanujan graphs with arbitrary degree and arbitrary number of vertices has been established by Marcus, Spielman and Srivastava \cite{MarcusSpielmanSrivastava2015,MarcusSpielmanSrivastava2015-2}.

\subsection{Eigenvalues of Random Graphs}\label{subsec:EigenvaluesAdjacencyRandomGraphs}

In the introduction, Theorem~\ref{thm:concentration-graph-eigenvalues} summarizes known results on the concentration of eigenvalues for random graphs $G(n,p)$. Here we want to explain the corresponding references in more detail.
For the normalized Laplacian the situation is simple:  Building on the results for the adjacency matrix and relating the spectrum of $\Lap(G(n,p))$ to that of $A(G(n,p))$, Coja-Oghlan \cite{CojaOghlan:2007gj} proved the result for the normalized Laplacian for probabilities $p\geq C \cdot \log(n)/n$ with a suitable constant $C$. For $p \gg (\log n)^2/n$ this was also shown by Chung, Lu and Vu \cite{Chung:2003wn}. A recent preprint by Hoffman, Kahle and Paquette \cite{Hoffman:2012} gives the precise result allowing all constants $C>1$ (and even $C >\frac{1}{2}$ when considering only the giant component of $G(n,p)$).

For the adjacency matrix the situation in the literature is more involved:
F\"uredi and Koml\'os~\cite{Furedi:1981ti} showed that for constant $p$ a.a.s. $\mu(G(n,p))=O(\sqrt{d})$, where $d=p(n-1)$ is the expected average degree. Their method of proof, the so-called \emph{trace method}, can be adapted to cover the range $\frac{\ln(n)^7}{n} \leq p \leq 1 - \frac{\ln(n)^7}{n}$ (see \cite{CojaOghlan:2005}).
Feige and Ofek~\cite{Feige:2005hp} extended the result to values of $p$ as small as $C\cdot \log n/n$, but their proof requires an upper bound on $p$. They used methods of Friedman, Kahn, and Szemer\'edi~\cite{Friedman:1989}, who proved that $\mu(G)=O(\sqrt{d})$ holds a.a.s.\ for \emph{random $d$-regular graphs} with constant $d$.
The most precise result is again by Hoffman, Kahle and Paquette \cite{Hoffman:2012}, who show that $\mu(G(n,p))=O(\sqrt{d})$ a.a.s. for $p\geq\gamma \log(n)/n$ for \emph{all} $\gamma>0$.

More precisely, in \cite{Hoffman:2012} it is shown that a.a.s. 
\begin{equation}\label{preciseStatementAdjacencyGraphs}
|\langle Ax,y\rangle| =O(\sqrt{d}) \text{ for all unit vectors } x,y \text{ such that } x\perp\1. 
\end{equation}
This, together with $\frac{1}{n}\langle A\1,\1\rangle=\frac{2|E|}{n} \in [d - O(\sqrt{d}), d +O(\sqrt{d})]$, which follows from a straight-forward application of a Chernoff bound, gives the result as stated in Theorem~\ref{thm:concentration-graph-eigenvalues} (see e.g. \cite[Lemma~2.1]{Feige:2005hp} or Lemma~\ref{LEMMAConditions} in this paper).

 We remark that both parts of Theorem~\ref{thm:concentration-graph-eigenvalues} can be extended to very sparse random graphs $G(n,p)$ with $p=\Theta(1/n)$ (for which they fail to hold as stated) by passing to a suitable large \emph{core subgraph}, see \cite{CojaOghlan:2007gj,Feige:2005hp,Hoffman:2012}. Moreover, analogous results are also known for other random graph models, including random $d$-regular graphs (see above) and 
random graphs with prescribed expected degree sequences \cite{Chung:2003wn,CojaOghlan:2009ud}.

\subsection{Simplicial Complexes and Cohomology}

A (finite, abstract) \emph{simplicial complex} $X$ is a finite set system that is closed under taking subsets, i.e.\ $F \subseteq G \in X$ implies $F \in X$. The sets in $X$ are called \emph{simplices} or \emph{faces} of $X$. The \emph{dimension} of a face $F$ is $\dim(F):=|F|-1$. We denote the set of $i$-dimensional faces of $X$ by $X_i$. The dimension of $X$ is the maximum dimension of any of its faces.
The $0$\mbox{-}dimensional faces are called \emph{vertices}. Formally, these are singletons (one-element sets) but in this context we will usually identify the singleton $\{v\}$ with its unique element $v$.

A $k$-dimensional simplicial complex is \emph{pure} if all maximal simplices in $X$ have dimension $k$. We define the \emph{degree} of a face $F$ as $\deg(F)=|\{G \in X_k : F \subseteq G\}|$. The \emph{link} of $F$ in $X$ is $\lk(F,X):=\{G \in X \colon F \cup G \in X, F \cap G=\emptyset\}$.
We denote by $K_n^k$ the  \emph{complete $k$-dimensional complex} on $n$ vertices, i.e.
$K_n^k = \{F \subseteq [n]: |F| \leq k+1\}.$

\subsubsection*{Orientations and Incidence Numbers} Throughout we assume that we have fixed a linear ordering on the vertex set $V:=X_0$ of $X$, and we consider the faces of $X$ with the orientations given by the order of their vertices. Formally, consider an $i$-simplex $F=\{v_0, v_1 ,\ldots, v_{i}\} \in X_i$, where $v_0<v_1<\ldots <v_i$.  For an $(i-1)$-simplex $G\in X_{i-1}$, we define the \emph{oriented incidence number} $[F:G]$ by setting $[F:G]:=(-1)^j$ if $G\subseteq F$ and $F\setminus G=\{v_j\}$, $0\leq j\leq i$, and $[F:G]:=0$ if $G\not\subseteq F$. In particular, for every vertex $v\in X_0$ and the unique empty face $\emptyset\in X_{-1}$, we have $[v:\emptyset]=1$.

\subsubsection*{Cohomology}
Let $X$ be a finite simplicial complex and let $\G$ be an Abelian group (we will mostly be concerned with the cases $\G=\Z_2$ and $\G=\R$, respectively). We denote by $C^i(X;\G)$ the group $\G^{X_i}$ of functions from $X_i$ to $\G$, which are called \emph{$i$-dimensional cochains of $X$ with coefficients in $\G$}. In particular, since $\emptyset$ is the unique empty face of $X$, we have $C^{-1}(X;\G)\cong \G$. It is convenient to define $C^i(X;\G):=0$ for $i<-1$ or $i>\dim X$.
The characteristic functions $e_F$ of faces $F\in X_i$ form a basis of $C^i(X;\G)$. They are called \emph{elementary cochains}. 

The \emph{coboundary map} $\delta_i\colon C^i(X;\G) \rightarrow C^{i+1}(X,\G)$ is the linear map given by 
\[
(\delta_i f)(F) := \sum_{G \in X_i} [F:G]\cdot  f(G) 
\]
for $f\in C^i(X;\G)$, $-1\leq i<\dim X$, and $\delta_i=0$ otherwise.

It is an easy but central observation that the composition $\delta_i\circ \delta_{i-1}=0$, which means that $B^i(X;\G):=\im \delta_{i-1} \subseteq Z^i(X;\G):=\ker \delta_i$. The elements of $B^i(X;\G)$ and $Z^i(X;\G)$ are called $i$-dimensional \emph{coboundaries} and \emph{cocycles}, respectively. Since $B^i(X;\G) \subseteq Z^i(X;\G)$, we can form the quotient group $\tilde{H}^i(X;\G):=Z^i(X;\G)/B^i(X;\G)$, the $i$-th (reduced) \emph{cohomology group} of $X$ with coefficients in $\G$.

\subsection{Norms on Cochains and Expansion}
\label{sec:expansion}

We now describe a very general definition of \emph{expansion} for simplicial complexes, which was introduced in \cite{Gromov:SingularitiesExpanders2-2010} (with a slightly different normalization and under the name \emph{inverse \textup{(}co\textup{)}filling norm}). 

Let $X$ be a finite simplicial complex. Assume that every cochain group $C^i(X;\G)$ is equipped with a \emph{pseudonorm} $\|\cdot \|$, taking real values and satisfying $\|f\|=\|-f\|$ and $\|f+g\|\leq \|f\|+\|g\|$ for all $f,g\in C^i(X;\G)$. We will focus on the following two cases.
\begin{enumerate} 
\item \textbf{$\R$-cochains with weighted $\ell_2$-norm:} Assume that we are given a \emph{weight function $w$} with nonnegative real values on the simplices of $X$. Define by $\langle f,g\rangle :=\sum_{F\in X_i} w(F)f(F)g(F)$ a weighted inner product on $C^i(X;\R)$. Observe that the inner products obtained in this way are characterized by the condition that the elementary cochains be pairwise orthogonal.
We then consider the corresponding weighted $\ell_2$-norm
$\|f\|=\|f\|_2:=\sqrt{\langle f,f\rangle}.$
\item \textbf{$\Z_2$-cochains with weighted Hamming norm:} Let $w$ be as before and define the \emph{weighted Hamming norm} on $C^i(X;\Z_2)$ by $\|f\|:=\sum_{F\in X_i:f(F)=1}w(F).$
\end{enumerate}

The idea is to define a notion of $i$-dimensional expansion that provides lower bounds for the norm of the coboundary $\delta_{i-1}(f)\in C^i(X;\G)$ of $(i-1)$-dimensional cochains $f\in C^{i-1}(X;\G)$. However, we cannot define such a lower bound in terms of the norm $\|f\|$ of $f$, since the set $B^{i-1}(X;\G)$ is always contained in the kernel of the coboundary operator $\delta=\delta_{i-1}$. Thus, the right comparison measure is the \emph{distance} of a cochain $f$ from this \emph{trivial part of the kernel}. That is, we define, for $f\in C^{i-1}(X;\G)$,
$$\|[f]\|:=\min\{\|f+\delta_{i-2}g\|\colon g\in C^{i-2}(X;\G)\}.$$

\subsubsection*{Coboundary Expansion for Arbitrary Coefficients}\label{def:face-expansion}
Suppose every cochain group $C^i(X;\G)$ is equipped with a pseudonorm $\|\cdot\|$ as above. We say that $X$ is \emph{$\varepsilon$-expanding in dimension $i$ } 
(with respect to 
$\G$ and the given norm) 
if 
$$
\|\delta f \| \geq \varepsilon \cdot \|[f]\|
$$
for all $f \in C^{i-1}(X;\G)$. The best possible $\varepsilon$ is called the $i$-dimensional expansion of $X$. Note that, in particular, $\tilde{H}^{i-1}(X;\G)=0$ if $X$ has $i$-dimensional expansion $\varepsilon>0$.

For an infinite family of $k$-dimensional complexes $(X_n)_{n \in \N}$ (where $k$ is fixed and independent of $n$) we say that the family $(X_n)$ is \emph{expanding in dimension $i$} (with respect to $\G$ and the given norm) if the $i$-dimensional expansion of all $X_n$ is bounded away from zero.

\subsubsection*{$\Z_2$-Coboundary Expansion}
Now we focus on the case of $\Z_2$-coefficients. Define a weight function by $w(F):=1/|X_i|$ for $F\in X_i$ (whenever $|X_i|>0$). In this setting, the normalized Hamming weight of a $\Z_2$-cochain $f\in C^{i-1}(X;\Z_2)$ is just the number of faces in the support of $f$ divided by the number of all $(i-1)$-faces of $X$.

If $X$ is is $\varepsilon$-expanding in dimension $i$ with respect to this norm, we also say that $X$ is \emph{$\Z_2$-coboundary $\varepsilon$-expanding} in dimension $i$.

Note that in the case $i=1$ of graphs, there are just two $0$-dimensional coboundaries, namely the constant functions ${\bf 0}$ and $\1$ on the set $V=X_0$ of vertices. Moreover, a $0$-dimensional cochain $f\in C^0(X;\Z_2)$ is in bijective correspondence with its support $S=\{v\in V\colon f(v)=1\}\subseteq V$, and $\|[f]\|=\frac{\min\{|S|,|V\setminus S|\}}{|V|}$. Thus, $1$-dimensional $\Z_2$-coboundary 
expansion corresponds precisely to the definition (\ref{eq:edge-expansion}) of edge expansion discussed in the introduction.

A basic observation in this context is that complete complexes are $\Z_2$-coboundary expanding in all dimensions. 
This was observed independently by Gromov~\cite{Gromov:SingularitiesExpanders2-2010}, Linial, Meshulam and Wallach~\cite{LinialMeshulam:HomologicalConnectivityRandom2Complexes-2006,MeshulamWallach:HomologicalConnectivityRandomComplexes-2009} and Newman and Rabinovich~\cite{Newman:2011}:
\begin{proposition}
 \label{prop:gromov}
The complete complex $K^k_n$ has $i$-dimen\-sional  $\Z_2$-coboundary 
expansion $1$ for all $i \in \{0,1,\ldots, k\}$. 
\end{proposition}
From this, standard Chernoff bounds immediately imply that a.a.s., $X^k(n,p)$ is $\Z_2$-coboundary expanding in dimension $k$ and $H^{k-1}(X^k(n,p);\Z_2)=0$ if $p> C\log n/n$ for a suitable constant $C$.
Much of the work in \cite{LinialMeshulam:HomologicalConnectivityRandom2Complexes-2006,MeshulamWallach:HomologicalConnectivityRandomComplexes-2009} is devoted to refining this argument to obtain the optimal constant $C=k$ for the threshold. 

Dotterrer and Kahle~\cite{Dotterrer:2010fk} prove results analogous to Proposition~\ref{prop:gromov} for some other complexes, specifically for skeleta of crosspolytopes and for complete multipartite complexes. They also explicitly raise the question whether there is some higher-dimensional analogue of the Cheeger inequality.
The most straightforward attempt at such an inequality would be to relate $\Z_2$-coboundary expansion and eigenvalue gaps of higher-dimensional Laplacians, which we discuss next. 

\subsection{Matrices and their spectra}
A symmetric real ($n\times n$)-matrix has a multiset of $n$ real eigenvalues, called its \emph{spectrum}, and $\R^n$ has an orthonormal basis of corresponding eigenvectors.

We recall the variational characterization of eigenvalues:
\begin{theorem}[Courant-Fischer Theorem, see e.g.\ {\cite[Theorem 4.2.11]{HornJohnson}}]\label{CourantFischer}
Let $M \in \R^{n \times n}$ be a symmetric matrix with eigenvalues $\lambda_1 \leq \lambda_2 \leq \ldots \leq \lambda_n$, and let $k$ be a given integer with $1\leq k \leq n$. Then
$$
\lambda_k = \min_{w_1,w_2,\ldots,w_{n-k} \in \R^n} \max_{\substack{x \neq 0, x \in \R^n\\x \perp w_1,w_2,\ldots,w_{n-k}}} \frac{\langle M x, x \rangle}{\langle x, x \rangle}
$$
and
$$
\lambda_k = \max_{w_1,w_2,\ldots,w_{k-1} \in \R^n} \min_{\substack{x \neq 0, x \in \R^n\\x \perp w_1,w_2,\ldots,w_{k-1}}} \frac{\langle M x, x \rangle}{\langle x, x \rangle}.
$$
\end{theorem}

For a matrix $M$ we denote ist $\ell_2$-norm by $\|M\| = \max_{x\neq 0}\|Mx\|/\|x\|$, which for a symmetric matrix $M$ equals the in absolute value largest eigenvalue of $M$.

\subsection{Higher-Dimensional Laplacians and Adjacency Matrices}
\label{sec:matrices-complexes}
We introduce generalizations of the graph Laplacians and the adjacency matrix for a $k$-di\-men\-sion\-al complex in all dimensions $0 \leq i \leq k-1$. Later on, we will only be concerned with these matrices in dimension $k-1$.
\subsubsection*{Adjacency matrices}
For a finite $k$-dimensional simplicial complex $X$ and $0 \leq i \leq k-1$ we define the \emph{adjacency matrix} $A_i=A_i(X)$ by
\[
 (A_i(X))_{F,G} = \begin{cases}
                -[F \cup G:F][F\cup G:G] = [F:F \cap G][G:F\cap G]& \text{if $F \sim G$},\\
		        0& \text{otherwise,}
               \end{cases}
\]
where $F,G \in X_i$ and we write $F \sim G$ if $F$ and $G$ share a common $(i-1)$-face $F\cap G$ and $F \cup G \in X_{i+1}$.
Figure~\ref{fig:Adjacency2D} illustrates the case $i=1$. An entry $A_1(X)_{e,e'}$ is non-zero exactly if the two edges $e$ and $e'$ share a common vertex and the triangle $e \cup e'$ is contained in $X$. The sign of $A_1(X)_{e,e'}$ is then determined by the orientations of the two edges.
\begin{figure}[htbp]
  \centering\includegraphics{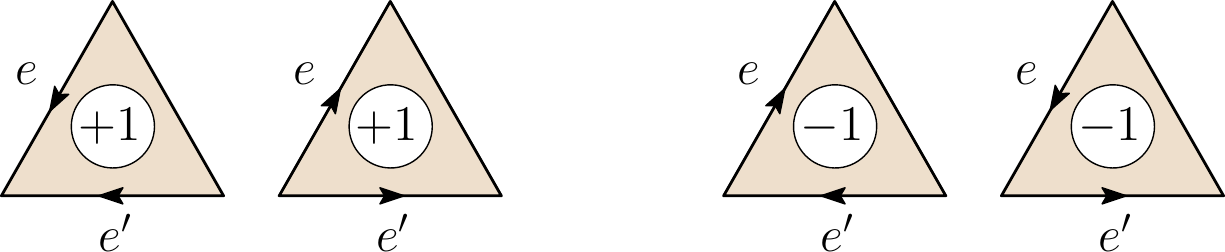}
  \caption{Signs of non-zero entries $A_1(X)_{e,e'}$. The arrows represent the orientations of edges. \label{fig:Adjacency2D}}
\end{figure}

Note that the matrix $A_0(X)$ agrees with the adjacency matrix of the graph $(X_0,X_1)$ because
$[\{u,v\}\!:\!u][\{u,v\}\!:\!v] = -1$ for all vertices $u,v \in X_0$.  The motivation for the signs in higher dimensions will hopefully become clear later on.

\subsubsection*{Weighted Laplacians}
Following the exposition in \cite{HorakJost}, we begin by defining a general weighted Laplacian. Suppose we are given a nonnegative weight function $w$ on the faces of a finite simplicial complex $X$ and that the spaces 
$C^i(X;\R)$ are equipped with the weighted inner product and the corresponding weighted $\ell_2$-norm as described above.

The elementary cochains $e_F$, $F\in X_i$, form an orthogonal basis of $C^i(X;\R)$. With respect to these bases, the coboundary map $\delta_i\colon C^i(X;\R)\rightarrow C^{i+1}(X;\R)$ is given by the following $|X_{i+1}|\times |X_i|$-matrix (for which we abuse notation and again use the symbol $\delta$):
\[
 (\delta_i(X))_{F,G} = [F:G].
\]

Consider the \emph{transpose map} $\delta_i^\ast\colon C^{i+1}(X;\R)\rightarrow C^i(X;\R)$ of $\delta_i(X)$ with respect to the given inner product. This transpose is determined by the condition that $\langle \delta_i^\ast f,g\rangle=\langle f,\delta_i g\rangle$ for all $f\in C^{i+1}(X;\R)$ and $g\in C^{i}(X;\R)$. More explicitly, 
$$(\delta_i^\ast f) (G)=\sum_{F\in X_{i+1}} \frac{w(F)}{w(G)}[F:G]f(F)$$ 
for $f\in C^{i+1}(X;\R)$ and $G\in X_i$.

For example, in the case of unit weights $w(F)=1$ for all $F\in X$, we get the standard inner product on $C^i(X;\R)$, 
and $\delta_i^\ast=\partial_{i+1}$ coincides with the usual \emph{boundary map} given on elementary cochains by $\partial_{i+1}(e_F)=\sum_{G\in X_i} [F:G]e_G$, $F\in X_{i+1}$. 

In general, for arbitrary weights $w$ on $X$, we define the \emph{weighted Laplacian} by 
$$\wlap_i^\down:=\delta_{i-1}\delta_{i-1}^\ast,\qquad \wlap_i^\up:=\delta_i^\ast \delta_i,\qquad \wlap_i:=\wlap_i^\down+\wlap_i^\up.$$
Note that all three maps $\wlap_i^\down,\wlap_i^\up,\wlap_i$ are \emph{self-adjoint} and \emph{positive semidefinite} (with respect to the given weighted inner product) linear operators on $C^i(X;\R)$.

In general, setting $\mathcal{H}_i=\mathcal{H}_i(X;\R):=\ker \wlap_i=\ker \wlap_i^\down \cap \ker \wlap_i^\up = \ker \delta_{i-1}^\ast \cap Z^i(X;\R)$, one gets a \emph{Hodge decomposition} of $C^i(X;\R)$ into pairwise orthogonal subspaces 
\begin{equation}\label{eqn:hodge-decomp}
 C^i(X;\R)=\mathcal{H}_i\oplus B^i(X;\R)\oplus \im (\delta_{i}^\ast),
\end{equation}
(see \cite{Eckmann:HarmonischeFunktionenRandwertaufgabenKomplex-1945,HorakJost}); in particular, $\mathcal{H}_i\cong H^i(X;\R)$.

\subsubsection*{Spectra of $\wlap_i^\up$ and Spectral Expansion}
Observe that, trivially, $B^i(X;\R)\subseteq \ker \wlap_i^\up$. Thus, every $f\in B^i(X;\R)$ is an eigenvector of $\wlap_i^\up$ with eigenvalue zero. We call these the trivial eigenvectors of $\wlap_i^\up$ and the trivial part of its spectrum. Thus, the nontrivial eigenvalues of $\wlap_i^\up$ are, by definition, the eigenvalues of the restriction of $\wlap_i^\up$ to the orthogonal complement (with respect to the given weighted inner product) $(B^i(X;\R))^\bot$.

By the variational definition of eigenvalues,  the minimal nontrivial eigenvalue of $\wlap_i^\up$ is given by
$$\min_{f\bot B^i(X;\R)} \frac{\langle \wlap_i^\up f,f\rangle}{\langle f,f\rangle}=\min_{f\bot B^i(X;\R)}\frac{\|\delta_if\|^2}{\|f\|^2}.$$
Thus, we see that the minimal nontrivial eigenvalue of $\wlap_i^\up$ is at least $\varepsilon^2$ iff $X$ has $(i+1)$-dimensional expansion at least $\varepsilon$ with respect to the given weighted $\ell_2$-norms on  real cochains. In this case, we will also say that $X$ is \emph{spectrally expanding} in dimension $i$.

We focus on the operator $\wlap_i^\up$, more precisely we consider $\wlap_{k-1}^\up$ for $k$-dimensional complexes because it corresponds to  coboundary expansion with respect to real coefficients and the $\ell_2$-norm. 

The spectra of the other two maps are related: By the Hodge decomposition \eqref{eqn:hodge-decomp} the spectrum of $\wlap_i$ is determined by the spectra of $\wlap_i^\down$ and $\wlap_i^\up$.  For any linear map $A$, the spectra of $AA^{\ast}$ and $A^{\ast}A$ differ only in the multiplicity of $0$; in particular, this holds for the spectra of $\wlap_i^\up$ and $\wlap_{i+1}^\down$.
Nevertheless, as we cover only $\wlap_{k-1}^\up$ for $k$-dimensional complexes, our results do not yield corresponding statements on $\wlap_{k-1}$.

\subsubsection*{Combinatorial Laplacians} The combinatorial Laplacian $L_i=L_i^\down+L_i^\up$ corresponds to the special case of the standard inner product $\langle f,g\rangle = \sum_{f\in X_i}f(F)g(F)$, that is, the case of \emph{unit weights} $w(F)=1$ for all $F\in X$. Thus,  $L_i^\up=L^\up_i(X) = \partial_{i+1}\delta_i.$

Recall that the matrix corresponding to the coboundary map $\delta_i$ with respect to the orthogonal basis of elementary cochains is, by abuse of notation, also denoted by $\delta_i=\delta_i(X)$, and its transpose $\delta_i^T$ corresponds to the boundary map $\partial_{i+1}$. The combinatorial Laplacian $L^\up_i$ can be expressed as the matrix $\delta_i^T\delta_i$.

We can now motivate the signs in the definition of the adjacency matrix $A_i(X)$: Recall that for a graph $G$ the combinatorial Laplacian satisfies $L(G) = D(G) - A(G)$. If we let $D_i(X)$ denote the diagonal matrix with entry ${D_i}_{F,F}=|\{H \in X_{i+1}: F \subset H\}|$ for $F \in X_i$, we also have $L^\up_i(X) = D_i(X) - A_i(X)$.

\subsubsection*{Normalized Laplacians} Suppose that $X$ is a pure $k$-dimensional simplicial complex.
The normalized Laplacian $\Lap_i=\Lap_i^\down+\Lap_i^\up$ is the special case of the weighted Laplacian obtained by taking the weight function $w(F):=\deg(F)$. That is, the corresponding weighted inner product is 
$$\langle f,g\rangle=\sum_{F\in X_i}\deg(F)f(F)g(F).$$

Let $\delta_i^\ast$ be the adjoint of $\delta_i$ with respect to this weighted inner product. Thus, 
$$(\delta_i^\ast f) (G)=\sum_{F\in X_{i+1}} \frac{\deg(F)}{\deg(G)}[F:G]f(F).$$ 
Note that we have $\deg(F) > 0$ for every $F \in X$, since we assume that $X$ is pure.
The normalized Laplacian is then $\Lap_i^\up=\Lap_i^\up(X)=\delta_i^\ast\delta_i$.

With respect to the basis of elementary cochains, the map $\Lap^\up_i$ corresponds to  the matrix $W_i^{-1}\delta_i^TW_{i+1}\delta_i$, where $W_i(X)$ denotes the diagonal matrix with entry ${W_i}_{F,F}=\deg(F)$. 
As $W_{k-1} = D_{k-1}$ and $W_k = I$, for $i = k-1$ we can write $\Lap^\up_{k-1}$ as the matrix $D_{k-1}^{-1}L^\up_{k-1}=I-D_{k-1}^{-1} A_{k-1}$.

\subsubsection*{Eigenvalues of the Complete Complex}

As an example we consider the spectra of the three matrices $L^\up_{k-1}(K_n^k)$, $\Lap^\up_{k-1}(K_n^k)$ and $A_{k-1}(K_n^k)$ for the complete complex $K_n^k$.
First recall the following well-known (and easily verifiable) lemma:

\begin{lemma}\label{BasisBoundariesCoboundaries}
For a complex $X$ with complete $(k-1)$-skeleton, the space $B^{(k-1)}(X) = \im \delta_{k-2}$  has dimension $\binom{n-1}{k-1}$. A basis is given by  $\big\{\delta_{k-2} e_F \!:\!  1 \notin F \in \binom{[n]}{k-1}\big\}$.
For the complete complex $K_n^k$, the space $\im \delta_{k-1}^\ast(K_n^k)$ is $\binom{n-1}{k}$-dimensional and has $\big\{\delta_{k-1}^\ast e_F \!:\! 1 \in F \in \binom{[n]}{k+1}\big\}$ as a basis.
\end{lemma}

\begin{lemma}\label{EigenvaluesCompleteComplex}
 The eigenvalues of the combinatorial Laplacian $L^\up_{k-1}(K_n^k)$ are $0$ with multiplicity $\binom{n-1}{k-1}$ and $n$ with multiplicity $\binom{n-1}{k}$.
 The normalized Laplacian $\Lap^\up_{k-1}(K_n^k)$ has eigenvalues $0$ with multiplicity $\binom{n-1}{k-1}$ and $\frac{n}{n-k}$ with multiplicity $\binom{n-1}{k}$.
 The eigenvalues of $A_{k-1}(K_n^k)$ are $n-k$ with multiplicity $\binom{n-1}{k-1}$ and $-k$ with multiplicity $\binom{n-1}{k}$. 
\end{lemma}
\begin{proof}
Because $K_n^k$ is $(n-k)$-regular, it suffices to consider the spectrum of $L^\up_{k-1}(K_n^k)$.
The following equality is contained implicitly in \cite{Kalai} and follows from a straightforward calculation using the matrix representations of the Laplacians: 
\[
 L^\up_{k-1}(K_n^k) + L^\down_{k-1}(K_n^k) = nI. 
\]
 Any non-zero element of $\ker L^\down_{k-1}(K_n^k) = \ker \delta_{k-2}^\ast(K_n^k) = \im \delta_{k-1}^\ast(K_n^k)$ is hence an eigenvector of $L^\up_{k-1}$ with eigenvalue $n$. Naturally, any non-zero element of $\ker L^\up_{k-1}(K_n^k) = Z^{k-1}(K_n^k) = B^{k-1}(K_n^k)$ is an eigenvector of $L^\up_{k-1}$ with eigenvalue $0$.
By Lemma~\ref{BasisBoundariesCoboundaries} $\im \delta_{k-1}^\ast(K_n^k)$ and $B^{k-1}(K_n^k)$ have dimensions $\binom{n-1}{k}$ and $\binom{n-1}{k-1}$, respectively.
As these add up to $\binom{n}{k}$, the dimension of $C^{k-1}(K_n^k)$, we have determined the complete spectrum.
\end{proof}

\section{Garland's Estimate Revisited}
\label{sec:Garland}
In \cite{Garland} Garland studies the normalized Laplacian $\Lap^\up_i(X)$. His main result regards a conjecture of Serre's on the cohomology of certain groups. As a technical lemma, he proves a bound for the nontrivial eigenvalues of $\Lap^\up_i(X)$ in terms of the eigenvalues of the Laplacian on links of lower-dimensional faces (see also \cite{Borel} for a very clear exposition). 

We state the result for the case of $\Lap^\up_{k-1}(X)$ and the links of $(k-2)$-dimensional faces $F \in X_{k-2}$. In this case, $\lk F = \lk(F,X)$ is a graph and the normalized Laplacian $\Lap_{0}^\up(\lk F)$ agrees with the usual normalized graph Laplacian $\Lap(\lk F)$.
Furthermore, we show an analogous result for the generalized adjacency matrix $A_{k-1}(X)$.

For a combinatorial application of Garland's ideas (to clique complexes of graphs) see \cite{AharoniBergerMeshulam:eigenvaluesHomologyFlagComplexes-2005}. Garland's estimate was subsequently further strengthened and extended. In particular, \.{Z}uk \cite{Zuk:1996}
proved that if a $2$-dimensional complex $X$ satisfies $\lambda_2(\Lap(\lk(v,X)))>1/2$ for all vertex links, then the fundamental group of $X$ has \emph{Kazhdan's} \emph{Property} (\emph{T}).

\subsection*{Normalized Laplacian}
\begin{theorem}[\cite{Garland}, see also {\cite[Theorem~1.5,1.6]{Borel}}]\label{Garland}
Let $X$ be a pure $k$-dimensional complex and let $\Lap_{k-1}^\up = \Lap_{k-1}^\up(X)$ be its normalized Laplacian. Denote by $\langle,\rangle$ the weighted inner product on  $C^{k-1}(X;\R)$ that is defined by $\langle f,g\rangle=\sum_{F\in X_{k-1} }\deg(F)f(F)g(F)$.
Assume that for all $F \in X_{k-2}$ $$\lambda_{\min} \leq \lambda_2(\Lap(\lk F)) \leq \lambda_{n-k+1}(\Lap(\lk F)) \leq \lambda_{\max}.$$
 Then for all $f \in \orcomp{k-1}{X}$ (where the orthogonal complement is taken with respect to $\langle,\rangle$)
\[
 (1 + k\lambda_{\min} - k) \langle f,f \rangle  \leq  \langle \Lap_{k-1}^\up f,f \rangle \leq (1 + k\lambda_{\max} - k) \langle f,f \rangle.
\]
Hence, all nontrivial eigenvalues of $\Lap_{k-1}^\up$ on $\orcomp{k-1}{X}$ lie in $[1 +  k\lambda_{\min} - k, 1 + k\lambda_{\max} - k].$
\end{theorem}
We remark that Garland only states the lower bound. The upper bound follows directly from the proof, which we reproduce here in our notation.
The main idea of the proof is to present the normalized Laplacian as a sum of matrices each of which has non-zero entries only on the link of some $(k-2)$-face. These matrices then correspond to the Laplacians of the links.

For a pure $k$-dimensional simplicial complex $X$, fix a face $F \in X_{k-2}$ of dimension $k-2$. 
Let $\rho_F$ be the diagonal $|X_{k-1}|\times|X_{k-1}|$-matrix defined by
$$(\rho_F)_{G,H} = \begin{cases}
		1 & \text{if } G = H \text{ and } F \subset G,\\
		0& \text{otherwise.}
		\end{cases}
$$
We set $\Lap_{k-1}^{\up,F}(X):=\rho_F\Lap_{k-1}^{\up}(X)\rho_F$ and for $f \in C^{k-1}(X)$ furthermore define $f_F \in C^0(\lk F)$ by $f_F(\{u\}) = [F \cup \{u\}:F]f(F \cup \{u\}).$

\begin{lemma}\label{GarlandLemma}
Let $X$ be a pure $k$-dimensional complex.
\begin{itemize}
\item[a)] $\sum_{F \in X_{k-2}}\Lap_{k-1}^{\up,F}(X) = \Lap_{k-1}^\up(X) + (k-1)I$. 
\item[b)] For $u, v \in V(\lk F)$ let $F_u=F \cup \{u\}$ and $F_v=F \cup \{v\}$. Then
$(\Lap_{k-1}^{\up,F}(X))_{F _u,F_v} = [F_u:F][F_v:F](\Lap(\lk F))_{u,v}.$
So, for $f \in C^{k-1}(X)$, $\langle\Lap_{k-1}^{\up,F}(X)f,f\rangle= \langle\Lap(\lk F)f_F,f_F\rangle.$
\item[c)] If $f \in \orcomp{k-1}{X}$ then $f_F \in \1^\perp$.
\end{itemize}
\end{lemma}

\begin{proof}
\begin{itemize}
\item[a)] Observe that $\Lap_{k-1}^{\up,F}(X)$ is obtained by replacing by $0$ all entries of $\Lap_{k-1}^\up(X)$  that are contained in a row or column corresponding to some $G$ with $F\nsubseteq G$. The non-zero entries of $\Lap_{k-1}^\up(X)$ lie on the diagonal or correspond to faces $G, H \in X_{k-1}$ that share a common $(k-2)$-face and for which $G \cup H \in X_k$.
Hence, every non-zero entry $(\Lap_{k-1}^\up(X))_{G,H}$ with $G \neq H$ is contained in exactly one summand and the diagonal entries, which are $1$, are each contained in exactly $k$ summands.

\item[b)]
First consider $u \neq v$ with $F \cup \{u,v\} \in X$. 
Straightforward calculations show that $\deg_X(F_u)= \deg_{\lk F}(u)$ and that furthermore 
$[F_{u,v}:F_u][F_{u,v}:F _v] = -[F_u:F][F_v:F]$ where $F_{u,v}$ stands for $F\cup\{u,v\}$.
Hence,
\[
(\Lap_{k-1}^{\up,F}(X))_{F_u,F_v} = \frac{[F_{u,v}:F_u][F_{u,v}:F_v]}{\deg_X(F_u)}
  = -  \frac{[F_u:F][F_v:F]}{\deg_{\lk F}(u)}
 = [F_u:F][F_v:F](\Lap(\lk F))_{u,v}.
\]
If $F \cup \{u,v\} \notin X$, the corresponding entry is $0$ in both matrices.
For the diagonal entries we get
$$
(\Lap_{k-1}^{\up,F}(X))_{F_u,F_u} = 1 = [F_u:F][F_u:F]  \Lap(\lk F)_{u,u}.
$$

\item[c)] Let $f \in \orcomp{k-1}{X}$. Then 
$\sum_{G \in X_{k-1}}\deg(G)f(G)[G:F] = \langle f , \delta_{k-2}e_F \rangle = 0$
and therefore
\[
\langle f_F,\1\rangle = \sum_{v\in V(\lk F)}\deg_{\lk F}(v) f_F(\{v\}) =
\sum_{v\in V(\lk F)}\deg(F_v) [F_v:F]f(F_v) =0.
\]
\end{itemize}
\end{proof}

The statements of Lemma~\ref{GarlandLemma} can easily be combined to prove Garland's estimate:
\begin{proof}[Proof of Theorem~\ref{Garland}]
Let $f \in \orcomp{k-1}{X}$. Then
$$\langle \sum\nolimits_{F \in X_{k-2}}\Lap_{k-1}^{\up,F}(X)f,f \rangle =  \sum\nolimits_{F \in \mathcal{F}_f} \langle \Lap(\lk F)f_F,f_F \rangle,$$
where $\mathcal{F}_f = \{F \in X_{k-2} | F \subset G \text{ for some } G \text{ with } f(G)\neq 0\}$.
Now, since $f \in \orcomp{k-1}{X}$, we have $f_F \in \1^\perp$ and $f_F \neq 0$ for $F \in \mathcal{F}_f$. 
As furthermore $\sum_{F \in \mathcal{F}_f}\langle f_F, f_F \rangle =  k \langle f,f \rangle$,
$$k \lambda_{\min} \langle f,f \rangle \leq \langle \sum\nolimits_{F \in X_{k-2}}\Lap_{k-1}^{\up,F}(X)f,f \rangle \leq  k \lambda_{\max} \langle f,f \rangle.$$
By Lemma~\ref{GarlandLemma} we have furthermore
$$\langle \Lap_{k-1}^\up(X) f,f \rangle  = \langle\sum\nolimits_{F \in X_{k-2}}\Lap_{k-1}^{\up,F}(X) f,f \rangle - (k-1)\langle f,f \rangle,$$
which concludes the proof.
\end{proof}

\subsection*{Adjacency Matrix}
We now turn to the generalized adjacency matrix $A_{k-1}(X)$. The same methods as above can be applied to achieve a result of similar nature (Proposition~\ref{Prop:Conclusionzz}). However, this only enables us to cover vectors from $\orcomp{k-1}{X}$. Controlling the behaviour on this space sufficed for the normalized Laplacian, where $B^{k-1}(X)$ is always a subspace of the eigenspace of zero.  For the generalized adjacency matrix we know much less about its eigenspaces, in particular we do not know of any trivial eigenvalues.

This is analogous to the situation for graphs, where $\1$, the all-ones vector, which is known to be the first eigenvector of the Laplacian (with eigenvalue $0$), is not necessarily an eigenvector of the adjacency matrix.
In \cite{Feige:2005hp} Feige and Ofek, considering the adjacency matrix of random graphs $G(n,p)$, show that for  $p$ large enough the first eigenvector can in some sense be replaced by $\1$.
Following their strategy, we show that controlling the behaviour of the generalized adjacency matrix $A_{k-1}(X)$ on the two spaces $B^{k-1}(X)$ and $\orcomp{k-1}{X}$ suffices to give concentration results for the spectrum of $A_{k-1}(X)$.

The results of this section together will yield the following theorem which can be considered as an analogue of Garland's Theorem~\ref{Garland} for the generalized adjacency matrix $A_{k-1}(X)$.

\begin{theorem}\label{GarlandAdjacency}
\renewcommand{\labelenumi}{(\roman{enumi})} 
\renewcommand{\theenumi}{(\roman{enumi})} 
Let $X$ be a $k$-dimensional simplicial complex with $n$ vertices and complete $(k-1)$-skeleton and let $A_{k-1}=A_{k-1}(X)$ be its generalized adjacency matrix. Fix a positive value $d$ and let $u = (1/\sqrt{n-k+1})\1$. Suppose that we have for all $F \in X_{k-2}$:
\begin{enumerate}
\item $|\langle A(\lk F)u,u \rangle - d| \leq f(n)$,
\item $|\langle A(\lk F)u,w \rangle| \leq g(n)$ for all $w \bot \1$ with $\|w\|=1$ and
\item $|\langle A(\lk F)w,w \rangle| \leq h(n)$ for all $w \bot \1$ with $\|w\|=1$.
\end{enumerate}
Let $\varphi(n) = f(n) +g(n) + h(n)$. Then:
\renewcommand{\labelenumi}{(\alph{enumi})} 
\renewcommand{\theenumi}{(\alph{enumi})} 
\begin{enumerate}
\item $|\langle A_{k-1}b,b \rangle - d| \leq k \cdot \varphi(n)$ for all $b\in B^{k-1}(X)$ with $\|b\|=1$,\label{Conclusionbb}
\item $|\langle A_{k-1}b,z \rangle| \leq k \cdot \varphi(n)$ for all $z \in\orcomp{k-1}{X}$ and $b\in B^{k-1}(X)$ with $\|b\|=\|z\|=1$ and\label{Conclusionbz}
\item $|\langle A_{k-1}z,z \rangle| \leq k \cdot h(n)$ for all $z \in \orcomp{k-1}{X}$ with $\|z\|=1$.\label{Conclusionzz}
\end{enumerate}
Hence, the largest $\binom{n-1}{k-1}$ eigenvalues of $A_{k-1}$ lie in the interval $[d-k\varphi(n),d+2k\varphi(n)+kh(n)]$, and the remaining $\binom{n-1}{k}$ eigenvalues lie in the interval $[-k(\varphi(n)+h(n)),kh(n)]$.
\end{theorem}

The following lemma explains the connection of Conclusions \ref{Conclusionbb}, \ref{Conclusionbz} and \ref{Conclusionzz} with the spectrum of $A_{k-1}(X)$. It is a generalization of \cite[Lemma~2.1]{Feige:2005hp}, which gives the a corresponding statement for graphs and deals with a single vector $u$, here replaced by the subspace $\U$, and is then used with $u= \frac{1}{\sqrt{n}}\1$. We will use $\U = B^{k-1}(X)$.
Note that $B^{k-1}(X) = B^{k-1}(K_n^k)$ if $X$ has a complete $(k-1)$-skeleton.

\begin{lemma}\label{LEMMAConditions}
\renewcommand{\labelenumi}{(\roman{enumi})} 
\renewcommand{\theenumi}{(\roman{enumi})} 
Let $X$ be a $k$-dimensional simplicial complex with $n$ vertices and complete $(k-1)$-skeleton, let $A_{k-1}=A_{k-1}(X)$ be its generalized adjacency matrix and let $\U$ be an $\binom{n-1}{k-1}$-di\-men\-sion\-al subspace of $C^{k-1}(X)$.
Suppose we have:
\begin{enumerate}
\item $0 \leq f_1(n) \leq \langle A_{k-1}b,b \rangle \leq f_2(n)$ for all $b\in \U$ with $\|b\|=1$,\label{Assumptionbb}
\item $|\langle A_{k-1}b,z \rangle| \leq g(n)$ for all $z \in \U^\bot$ and $b\in \U$ with $\|b\|=\|z\|=1$ and\label{Assumptionbz}
\item $|\langle A_{k-1}z,z \rangle| \leq h(n)$ for all $z \in \U^\bot$ with $\|z\|=1$.\label{Assumptionzz}
\end{enumerate}
Then the largest $\binom{n-1}{k-1}$ eigenvalues of $A_{k-1}$ lie in the interval $[f_1(n),f_2(n)+g(n)+h(n)]$, and the remaining $\binom{n-1}{k}$ eigenvalues lie in the interval $[-(g(n)+h(n)),h(n)]$.
\end{lemma}

\begin{proof}[Proof of Lemma~\ref{LEMMAConditions}]
Write $A=A_{k-1}$.
Let $v$ be an arbitrary unit vector. Then there are unit vectors $b \in \U$, $z \in \U^\bot$ and $-1\leq \alpha, \beta \leq 1$ such that $v =\alpha b + \beta z$ and $\alpha^2 + \beta^2 =1$.
Because $A$ is symmetric, we get
$$\langle Av,v \rangle = \alpha^2\langle Ab,b\rangle + 2 \alpha\beta \langle Ab,z \rangle + \beta^2\langle Az,z \rangle.$$
Using \ref{Assumptionbb},\ref{Assumptionbz} and \ref{Assumptionzz} as well as $\alpha\beta \leq 1/2$ and $0\leq \alpha, \beta \leq 1$, we can conclude that
$$-g(n)-h(n)\leq \langle Av,v \rangle \leq f_2(n)+g(n)+h(n).$$
Hence, all eigenvalues of $A$ are contained in $[-g(n)-h(n),f_2(n)+g(n)+h(n)]$.
Now, let $\mu_1 \leq \mu_2 \leq \ldots \leq \mu_{\binom{n}{k}}$ be the eigenvalues of $A$.
Applying \ref{Assumptionbb} and \ref{Assumptionzz}  we get
$$\mu_{\binom{n-1}{k}} \leq \max_{z \in \U^\bot, \|z\|=1} \langle Az,z \rangle \leq h(n) \;\;\text{ and }\;\; \mu_{\binom{n-1}{k}+1} \geq \min_{b\in \U, \|b\|=1} \langle Ab,b \rangle \geq f_1(n),$$
by the variational characterization of eigenvalues (Theorem~\ref{CourantFischer}), since $\dim\U^\bot = \binom{n-1}{k}$.
\end{proof}

The proof of Theorem~\ref{GarlandAdjacency} makes up the remainder of this section and is divided into two parts. We first deal with Conclusion \ref{Conclusionzz} and then turn to Conclusions \ref{Conclusionbb} and \ref{Conclusionbz}.

\subsection*{Conclusion \ref{Conclusionzz} - Behaviour on $\orcomp{k-1}{X}$}
We address Conclusion \ref{Conclusionzz} with the same methods that we used to prove Garland's Theorem~\ref{Garland}.
\begin{proposition}\label{Prop:Conclusionzz}
Let $X$ be a $k$-dimensional complex and let $A_{k-1}=A_{k-1}(X)$ be its generalized adjacency matrix. 
Assume that for all $F \in X_{k-2}$ and for all $w \in C^0(\lk F)$ with $w \bot \1$
$$|\langle A(\lk F)w ,w\rangle| \leq h(n) \langle w,w\rangle.$$
 Then for all $z \in \orcomp{k-1}{X}$ (where the orthogonal complement is taken with respect to the standard, non-weighted inner product)
 $$|\langle A_{k-1} z,z \rangle| \leq k \cdot  h(n) \langle z,z \rangle.$$ 
\end{proposition}

\begin{proof}
For any face $F \in X_{k-2}$ set $A_{k-1}^F:=\rho_FA_{k-1}\rho_F$, the matrix obtained from $A_{k-1}$ by replacing all rows and columns corresponding to $(k-1)$-faces not containing $F$ by all-zero rows/columns.
Similar as in Lemma~\ref{GarlandLemma}, straightforward calculations show:
\begin{enumerate}
 \item[a)] $\sum_{F \in X_{k-2}}A_{k-1}^F = A_{k-1}$,
 \item[b)] $(A_{k-1}^F)_{F \cup \{u\},F \cup \{v\}} = [F \cup \{u\}:F][F \cup \{v\}:F]A(\lk F)_{u,v}$ for $F \in X_{k-2}$ and $u, v \in V(\lk F)$ 
 and hence $\langle A_{k-1}^F f,f \rangle = \langle A(\lk F) f_F,f_F \rangle$ for any $f \in C^{k-1}(X)$.
\end{enumerate} 
As $z\in \orcomp{k-1}{X}$ implies $z_F \in \1^\perp$ also with respect to the non-weighted inner product, this proves the proposition:
$$
|\langle A_{k-1} z,z \rangle| = |\sum_{F \in X_{k-2}} \langle A_{k-1}^F z,z \rangle| \leq \sum_{F \in X_{k-2}} |\langle A(\lk F) z_F,z_F \rangle| \leq  k \cdot  h(n) \langle z,z \rangle.
$$
\end{proof}

As explained above, in contrast to the Laplacian, for the adjacency matrix we are also interested in the behaviour on $B^{k-1}(X)$. For this space, we can not apply a proof similar to the one above because $f \in B^{k-1}(X)$ does not imply that $f^F$ is constant for every  $F \in X_{k-2}$. (For a $k$-dimensional complex with complete $(k-1)$-skeleton, the basis vectors $\delta_{k-2} e_F$ are a simple counterexample.)

\subsection*{Conclusions \ref{Conclusionbb} and \ref{Conclusionbz} - Behaviour on $B^{k-1}(X)$}
For $b \in B^{k-1}(X)$ we have $A_{k-1}(X)b = D_{k-1}(X)b$. If the complex $X$ was regular, i.e.~all $(k-1)$-faces would have the same degree $d$, $B^{k-1}(X)$ would be a subspace of the eigenspace of $d$.

The random complex $X^k(n,p)$ is not regular but with high probability the degrees of all $(k-1)$-faces lie close to the expected average degree $d=p(n-1)$. 
For an arbitrary complex we can fix any positive value $d$ and study the divergences of the degrees from $d$ by considering the diagonal matrix $\error(X) = D_{k-1}(X)-dI$ which has entries $\error(X)_{F,F} =  \deg_X(F) - d$. Then $A_{k-1}(X)b = \error(X) b + db$ for $b \in B^{k-1}(X)$.

It will turn out that our main task is to control the behaviour of $\|\error(X) b\|$ for all $b \in B^{k-1}(X)$.
We manage to reduce this to a question on the links of $(k-2)$-faces: Proposition~\ref{Prop:ReducingToLinks} relates $\|\error(X) b\|$ for every $b \in B^{k-1}(X)$ to the values $\|\error(X)\delta_{k-2} e_F\|$ for $F \in X_{k-2}$, to the behaviour of $\error(X)$ on the coboundaries of elementary cochains. These values in turn match the values $\|\error(\lk F)\1\|$ on the corresponding links.

\begin{proposition}\label{Prop:ReducingToLinks}
Let $X$ be a $k$-dimensional complex with vertex set $[n]$ and complete $(k-1)$-skeleton. Fix some positive value $d$ and let $\error = \error(X) = D_{k-1}(X)-dI$. 
Assume that for all $F \in X_{k-2}$ we have 
$$\|\error\delta e_F\| \leq f(n) \|\delta e_F\|.$$
 Then for all $b \in B^{k-1}(X)$ 
 $$\|\error b\| \leq k \cdot  f(n) \|b\|.$$ 
\end{proposition}

\begin{remark} Proposition~\ref{Prop:ReducingToLinks} also holds if $\error$ is replaced by any diagonal $|X_{k-1}|\times|X_{k-1}|$-matrix.
\end{remark}

The proof of Proposition~\ref{Prop:ReducingToLinks} is deferred to the end of this section. Here is how we use it to address Conclusions \ref{Conclusionbb} and \ref{Conclusionbz}.

\begin{proposition}\label{Prop:ConclusionbbConclusionbz}
\renewcommand{\labelenumi}{(\roman{enumi})} 
Let $X$ be a $k$-dimensional simplicial complex with $n$ vertices and complete $(k-1)$-skeleton. Fix some postive value $d$ and suppose that we have $$\sum_{v\in V(\lk F)}(\deg_{\lk(F)}(v)-d)^2 = \|\error(\lk F)\1\|^2 \leq f(n)^2 (n-k+1)$$
for all $F \in X_{k-2}$. Then
\begin{enumerate}
\item $|\langle A_{k-1}b,b \rangle-d| \leq k \cdot f(n)$ for all $b \in B^{k-1}(X)$ with $\|b\|=1$ and
\item $|\langle A_{k-1}b,z \rangle| \leq k \cdot f(n)$ for all $b \in B^{k-1}(X)$, $z \in \orcomp{k-1}{X}$ with $\|b\| = \|z\|=1$.
\end{enumerate}
\end{proposition}

\begin{proof}
As $\deg(F \cup \{v\}) = \deg_{\lk F}(v)$ for $v \notin F$, we have
$$\|\error\delta e_F\|^2 = \sum\nolimits_{H \supset F} (\deg(H) -d)^2 = \sum\nolimits_{v \notin F} (\deg_{\lk F}(v) -d)^2 \leq f(n)^2 (n-k+1).
$$
By Proposition~\ref{Prop:ReducingToLinks} we hence have $\|\error b\| \leq k \cdot f(n) \|b\|$ for all $b \in B^{k-1}(X)$.
Now, let $b \in B^{k-1}(X)$ and $z \in \orcomp{k-1}{X}$.
As $A_{k-1}b = D_{k-1}b = db + \error b,$
we get
$$|\langle A_{k-1}b,b \rangle -  d\|b\|^2 | \leq \|b\| \cdot \|\error b\| \leq k \cdot f(n)\|b\|^2$$ and $$ |\langle A_{k-1}b,z \rangle| \leq |\langle \error b,z\rangle| \leq \|z\| \cdot \|\error b\|  \leq k \cdot f(n)\|z\|\|b\|.$$
\end{proof}

To conclude the proof of Theorem~\ref{GarlandAdjacency} we are missing a small lemma:
\begin{lemma}
\renewcommand{\labelenumi}{(\roman{enumi})} 
Let $G$ be a graph with $n$ vertices with adjacency matrix $A=A(G)$ and let $u=\frac{1}{\sqrt{n}}\1$. Fix a positive value $d$.
Assume that 
\begin{enumerate}
\item $|\langle Au,u \rangle - d| \leq f(n)$,
\item $|\langle Au,w \rangle| \leq g(n)$ for all $w \bot \1$ with $\|w\|=1$ and
\item $|\langle Aw,w \rangle| \leq h(n)$ for all $w \bot \1$ with $\|w\|=1$.
\end{enumerate}
Then $\| \error(G) \1 \|^2 = \sum_{v\in V}(\deg(v)-d)^2 \leq (f(n) +g(n) + h(n))^2 n$. 
\end{lemma}
\begin{proof}
We have $\| \error(G) \1 \| = \| (\frac{d}{n}J -A)\1\| \leq \|\frac{d}{n}J -A\|\cdot\|\1\|$
and the conditions above imply  $\|\frac{d}{n}J -A\| \leq f(n) +g(n) + h(n)$. 
\end{proof}

\subsubsection*{Proof of Proposition~\ref{Prop:ReducingToLinks}}
The proof of Propositon~\ref{Prop:ReducingToLinks} is based on the observations in the following lemma.
Its proof will use the following simple consequence of the Cauchy-Schwarz inequality:
\begin{equation}
\left(\sum_{i\in I} a_i\right)^2 \leq |I| \sum_{i\in I} a_i^2. \label{CauchySchwarz}
\end{equation}

\begin{lemma}\label{ReducingToLinks}
Let $X$ be a $k$-complex with vertex set $[n]$ and complete $(k-1)$-skeleton and let $b \in B^{k-1}(X)$. For every $(k-2)$-face $F \in X_{k-2}$ define
$$h_b(F) := \sum_{v \notin F} [F \cup \{v\}:F] b(F \cup \{v\}).$$
Then
\begin{enumerate}
\item[a)] $b(H) = \frac{1}{n} \sum_{F \subset H, F \in X_{k-2}} [H:F] h_b(F)$ for $H \in X_{k-1}$,
\item[b)]  $\langle\error b,\error b\rangle \leq \frac{k}{n^2} \sum_{F \in X_{k-2}} h_b(F)^2 \langle\error \delta e_F,\error \delta e_F\rangle$,
\item[c)] $\sum_{F \in X_{k-2}} h_b(F)^2 \leq k(n-k+1)\langle b,b \rangle$.
\end{enumerate}
\end{lemma}

\begin{proof}
 \begin{enumerate}
  \item[a)] As $X$ has a complete $(k-1)$-skeleton, we have $b \in B^{k-1}(X)= B^{k-1}(K_n^k)$ and $\delta_{k-1}(K_n^k) b = 0$. Thus, for any $H \in X_{k-1}$ and $v \notin H$:
   $$
   0 = (\delta_{k-1}(K_n^k)b) (H \cup \{v\}) = [H \cup \{v\}: H] b(H) + \sum_{F \subset H} [H \cup \{v\}: F \cup \{v\}] b(F \cup \{v\}).
   $$
   Note that $- [H \cup \{v\}: H] [H \cup \{v\}: F \cup \{v\}] = [H:F][F \cup \{v\}:F]$. Thus, we can rearrange:
   $$
   b(H) = - [H \cup \{v\}: H] \sum_{F \subset H} [H \cup \{v\}: F \cup \{v\}]  b(F \cup \{v\}) =  \sum_{F \subset H} [H:F][F \cup \{v\}:F]  b(F \cup \{v\}).
   $$
   Summing over all $v \notin H$ and adding additional multiples of $b(H)$, we get
   \begin{multline*}
   n \cdot b(H) = \sum_{v \notin H} \sum_{F \subset H} [H:F][F \cup \{v\}:F] b(F \cup \{v\}) + k \cdot b(H)\\
   =  \sum_{F \subset H} [H:F] \sum_{v \notin F} [F \cup \{v\}:F] b(F \cup \{v\})
   = \sum_{F \subset H} [H:F] h_b(F).
   \end{multline*}
  \item[b)] By a) and inequality~\eqref{CauchySchwarz} and because $\langle\error \delta e_F,\error \delta e_F\rangle = \sum_{H \supset F} \error (H)^2$ for $F \in X_{k-2}$: 
   \begin{multline*}
   \langle\error b,\error b\rangle = \sum_{H \in X_{k-1}} \error (H)^2 b(H)^2
   = \frac{1}{n^2} \sum_{H \in X_{k-1}} \error (H)^2 \left(\sum_{F \subset H} [H:F] h_b(F)\right)^2\\
   \leq \frac{k}{n^2} \sum_{H \in X_{k-1}} \error (H)^2 \sum_{F \subset H} h_b(F)^2
   = \frac{k}{n^2} \sum_{F \in X_{k-2}} h_b(F)^2 \langle\error \delta e_F,\error \delta e_F\rangle.
   \end{multline*}
  \item[c)] Again by inequality~\eqref{CauchySchwarz}:
   \begin{multline*}
    \sum_{F \in X_{k-2}} h_b(F)^2  \leq \sum_{F \in X_{k-2}} (n-k+1) \cdot \sum_{v \notin F} b(F \cup \{v\})^2\\
    = (n-k+1) \cdot \sum_{H \in X_{k-1}} k \cdot b(H)^2 = k(n-k+1)\langle b,b \rangle.
   \end{multline*}
 \end{enumerate}
\end{proof}
The statements of Lemma~\ref{ReducingToLinks} together yield Proposition~\ref{Prop:ReducingToLinks}:
\begin{proof}[Proof of Propositon~\ref{Prop:ReducingToLinks}]
Let $b \in B^{k-1}(X)$. As $\|\delta e_F\| = \sqrt{n-k+1}$ for $F \in X_{k-2}$, by Lemma~\ref{ReducingToLinks}:
   \begin{multline*}
   \langle\error b,\error b\rangle \leq \frac{k}{n^2} \sum_{F \in X_{k-2}} h_b(F)^2 \langle\error \delta e_F,\error \delta e_F\rangle \leq \frac{k}{n^2} \sum_{F \in X_{k-2}} h_b(F)^2 f(n)^2 \langle\delta e_F,\delta e_F\rangle\\
   \leq k^2\cdot \frac{(n-k+1)^2}{n^2} \cdot f(n)^2  \langle b,b \rangle \leq k^2 \cdot f(n)^2 \langle b,b \rangle.
   \end{multline*}
\end{proof}

\section{The Spectra of Random Complexes}\label{sec:EigenvaluesRandomComplexes}

In this section, we prove Theorem~\ref{EigenvaluesRandomComplexes}, the concentration result on the spectra of the normalized Laplacian and the generalized adjacency matrix of random complexes $X^k(n,p)$.
The basic idea is to reduce the statement to a question on the links of $(k-2)$-faces by applying Theorems~\ref{Garland} and \ref{GarlandAdjacency}. Since for every $(k-2)$-face $F$, the link $\lk(F,X^k(n,p))$ is a random graph with the same distribution as $G(n-k+1,p)$, we can then apply results on the eigenvalues of random graphs. For convenience, we repeat Theorem~\ref{EigenvaluesRandomComplexes}:
{%
\renewcommand\thetheorem{\savedtheoremnumber}%
\begin{theorem}
Let $k\geq2$. For every $c>0$ and every $\gamma > c$ there exists a constant $C>0$ with the following property:
Assume  $p \geq (k+\gamma)\log(n)/n$ and let $d:= p(n-k)$. 
Then for $\gamma_A=C\cdot\sqrt{d}$ and $\gamma_\Delta=C/\sqrt{d}$ the following statements hold with probability at least $1-n^{-c}$:
\begin{enumerate}
\item[\textup{(i)}] The largest $\binom{n-1}{k-1}$ eigenvalues of $A_{k-1}(X^k(n,p))$ lie in the interval $[d-\gamma_A,d+\gamma_A]$, and the remaining $\binom{n-1}{k}$ eigenvalues lie in the interval $[-\gamma_A,+\gamma_A]$.
\item[\textup{(ii)}] The smallest $\binom{n-1}{k-1}$ eigenvalues of $\Lap_{k-1}^\up(X^k(n,p))$ are \textup{(}trivially\textup{)} zero, and the remaining $\binom{n-1}{k}$ eigenvalues lie in the interval $[1-\gamma_\Delta,1+\gamma_\Delta]$. In particular, $\tilde{H}^{k-1}(X^k(n,p);\R)=0$.
\end{enumerate} 
For the adjacency matrix \textup{(i)} even holds for $p\geq \gamma\cdot \log n/n$.
\end{theorem}
\addtocounter{theorem}{-1}%
}%

Observe that $B^{k-1}(K_n^k)\subseteq \ker \Lap^\up_{k-1}(X^k(n,p))$ because $X^k(n,p)$ has a complete $(k-1)$-skeleton, so the multiplicity of $0$ as an eigenvalue of $ \Lap^\up_{k-1}(X^k(n,p))$ is at least $\binom{n-1}{k-1}$.

\begin{proof}[Proof of Theorem~\ref{EigenvaluesRandomComplexes}]
Let $c>0$ and let $\gamma>c$. For $F \in \binom{[n]}{k-1}$, the link $\lk F = \lk(F,X^k(n,p))$ is a random graph $G(n-k+1,p)$. By Theorem~\ref{thm:concentration-graph-eigenvalues} (and \eqref{preciseStatementAdjacencyGraphs} in Section~\ref{subsec:EigenvaluesAdjacencyRandomGraphs}) we can hence choose $C>0$ such that for $p \geq (k+\gamma)\log(n)/n$ the following holds with probability at least $1-n^{-c-k+1}$:
\begin{enumerate}
\item[\textup{(i)}] $|\langle Ax,y\rangle| \leq C \sqrt{d}$ for all unit vectors $x,y$ with $x\perp\1$ and $\frac{1}{n-k+1}\langle A\1,\1\rangle\in [d - C\sqrt{d}, d +C\sqrt{d}]$.
\item[\textup{(ii)}] All nontrivial eigenvalues of $\Delta(\lk F)$ are contained in the interval $[1-C/(k\sqrt{d}),1+C/(k\sqrt{d})]$.
\end{enumerate} 
We first focus on the adjacency matrix:
A union bound yields that for $p \geq (k+\gamma)\log(n)/n$:
\[
\Pr\left[\exists F \in X_{k-2}:|\tfrac{1}{n-k+1}\langle A\1,\1\rangle-d| > C\sqrt{d} \;\text{ or }\;|\langle Ax,y\rangle| > C \sqrt{d}\text{ for some } x\perp\1,y\right] \leq n^{-c}.
\]
This implies that the conditions of Theorem~\ref{GarlandAdjacency} with $f(n),g(n),h(n)=O(\sqrt{d})$, and hence the desired concentration bounds, are fulfilled with probability at least $1-n^{-c}$.
Note that so far by Theorem~\ref{thm:concentration-graph-eigenvalues} it would have sufficed to choose $p\geq \gamma\log(n)/n$.

Now consider the normalized Laplacian. Again, a union bound gives for $p \geq (k+\gamma)\log(n)/n$
$$\Pr\left[ \forall F \in X_{k-2}: 1-C/(k\sqrt{d}) \leq \lambda_2(\Delta(\lk F)) \leq \lambda_{n-k+1}(\Delta(\lk F)) \leq 1+C/(k\sqrt{d})]\right] \geq 1 - n^{-c}.$$
 
For every $(k-1)$-face $H \in \tbinom{[n]}{k}$ of $X^k(n,p)$, the random variable $\deg(H)$ is binomially distributed with parameters $(n-k)$ and $p$. So, for $n$ large enough, the complex $X^k(n,p)$ is pure with probability at least $1-n^{-c}$.
Hence, also the conditions of Theorem~\ref{Garland} are fulfilled with probability at least $1-n^{-c}$.
\end{proof}

\begin{remark}\label{rem:ConditionsEigenvaluesLaplacians}
Note that that the preceding proof works for any random distribution $\mathcal{X}_k(n,p)$ on $k$-dimensional simplicial complexes with $n$ vertices and complete $(k-1)$-skeleton with the property that the link $\lk(F,\mathcal{X}_k(n,p))$ of every $F \in \binom{[n]}{k-1}$ is a random graph with distribution $G(n-k+1,p)$. 
\end{remark}

\section{Spectral vs. Coboundary Expansion}
\label{sec:CheegerCounterexample}
In this section, we prove Theorem~\ref{thm:counterexample}. As mentioned in the introduction, the examples are obtained by a probabilistic construction.

\subsubsection*{Basic Construction} Denote by $Y^k(n,p)$ the random $k$-dimensional simplicial complex with vertex set $V=[n]$ and complete $(k-1)$-skeleton obtained as follows: Randomly choose a map $a\colon  \binom{V}{k}\rightarrow \Z_2$ by setting $a(F)=1$ with probability $1/2$ and $a(F)=0$ otherwise, independently for each $F\in \binom{V}{k}$. Thus, the support of $a$ has the same distribution as the $(k-1)$-faces of the Linial-Meshulam random complex $X^{k-1}(n,1/2)$.

Call $H\in \binom{V}{k+1}$ ``\emph{good}'' iff $H$ contains an \emph{even} number of $(k-1)$-faces $F$ with $a(F)=1$. Every good $H$ is added as a $k$-face to $Y^{k}(n,p)$ independently with probability $p$.
Note that, by construction, $a$ is a $\Z_2$-\emph{cocycle} in the complex $Y^k(n,p)$, i.e., $a\in Z^{k-1}(Y^k(n,p);\Z_2)$.

For any fixed $b\in C^{k-1}(Y^k(n,p);\Z_2)=\Z_2^{\binom{V}{k}}$, the expected normalized Hamming distance between $b$ and the randomly chosen $a$ equals $1/2$. Since there are fewer than $2^{\binom{n}{k-1}}$ \emph{coboundaries} $b\in B^{k-1}(Y^k(n,p);\Z_2)$ and $\binom{n}{k}$ independent random choices for the entries of $a$, a straightforward application of a Chernoff bound (see, e.g., \cite[Theorem~1]{Janson}, \cite[Theorem~2.1]{JLR}) plus a union bound implies that, a.a.s.,  
$a$ has normalized Hamming distance $1/2-o(1)$ from any coboundary, i.e., 
\[
\|[a]\|\geq 1/2-o(1).
\]
In particular, a.a.s.\  $\tilde{H}^{k-1}(Y^k(n,p),\Z_2)\neq 0$. 

Note that  for $H \in \binom{V}{k+1}$, the probability that $H$ is a $k$-face of $Y^k(n,p)$ equals $p/2$. However, in contrast to the model $X^k(n,p/2)$, the decisions for different $k$-faces that share some $(k-1)$-face are not independent. Nevertheless, we can still easily analyze the links of $(k-2)$-faces in $Y^k(n,p)$:

\begin{lemma}\label{LinkIndependence}
For every $(k-2)$-face $H \in (Y^k(n,p))_{k-2}=\binom{V}{k-1}$, the random graph $\lk(H,Y^k(n,p))$ has the distribution $G(n-k+1,p/2)$.
\end{lemma}

\begin{proof}
First note that it suffices to consider the case $p=1$, because $\lk(H,Y^k(n,p))$ carries the distribution attained by taking every edge in $\lk(H,Y^k(n,1))$ independently with probability $p$.

For simplicity, we write $Y$ instead of $Y^k(n,1)$.
Let $U:=V\setminus H$. For $e\in \binom{U}{2}$, consider the event that $e\in \lk(H,Y)$, i.e., that $H\cup e\in Y$. We need to show that these events are mutually independent. To see this, choose and fix, for each $e\in \binom{U}{2}$, an arbitrary $(k-1)$-simplex $F_e$ with $e\subseteq F_e \subseteq H\cup e$; we call these the ``\emph{undecided}'' $(k-1)$-simplices, and let $\mathcal{D}:=\binom{V}{k}\setminus \{F_e\colon e\in \binom{U}{2}\}$ be the set of remaining, ``\emph{decided}'' $(k-1)$-simplices. Note that, by construction, each $k$-simplex of the form $H\cup e$, $e\in \binom{U}{2}$, contains exactly one undecided $(k-1)$-simplex $F_e$ and that these are pairwise distinct. Fix a map $r\colon \mathcal{D} \rightarrow \Z_2$ and condition upon the event that $r$ is the restriction of $a$ to $\mathcal{D}$. For each $e\in \binom{U}{2}$, we have $e\in \lk(H,Y)$ iff $a(F_e)=\sum_{F\in \mathcal{D}, F\subset H\cup e} r(F)$. For a fixed $r$, the (conditional) probability of this happening is $1/2$, and the values $a(F_e)$ are mutually independent since the $F_e$ are pairwise distinct. Thus, for any set of edges $e_1,\ldots, e_\ell\in \binom{U}{2}$ and for any fixed $r$, we get the conditional probability
$\Pr[\forall i: e_i \in \lk(H,Y)\mid a|_\mathcal{D}=r]=(1/2)^\ell.$
Since this holds for all choices of $r$, it also holds 
unconditionally, which proves the lemma.
\end{proof}
For $p \geq (k+\gamma) \log(n)/n$ with $\gamma>0$ we can thus, by this lemma and Remark~\ref{rem:ConditionsEigenvaluesLaplacians}, proceed as in the proof of Theorem~\ref{EigenvaluesRandomComplexes} to show that there exists $\gamma_\Delta=O(1/\sqrt{pn})$ such that a.a.s.~the nontrivial part of the spectrum of $\Lap_{k-1}^\up(Y^k(n,p))$ lies in the interval $[1-\gamma_\Delta,1+\gamma_\Delta]$.

\subsubsection*{Modification} We have so far shown the existence of an infinite family of $k$-dimensional complexes that is spectrally but not $\Z_2$-coboundary expanding.
However, the complexes constructed have non-trivial cohomology groups $\tilde{H}^{k-1}(Y,\Z_2)$, and hence also $\tilde{H}_{k-1}(Y,\Z) \neq 0$, because $a$ is a $\Z_2$-cocycle by construction.

To change this we can add a second round to our experiment and randomly add possible further $k$-simplices as follows: After constructing $Y^k(n,p)$, we add each $H \in \binom{V}{k+1}$ independently with some probability $q$. We denote the obtained  random complex by $Z^k(n,p,q)$. Thus, $Z^k(n,p,q)$ is the union of $Y^k(n,p)$ and the Linial-Meshulam random complex $X^k(n,q)$.
We assume that $p,q \geq C\cdot \log(n)/n$ for some suitably chosen $C$.

To analyze the $\Z_2$-coboundary expansion of $Z=Z^k(n,p,q)$, we first argue that $Z$, a.a.s., contains at least $\frac{p}{2} (1-o(1)) \binom{n}{k+1}$ many $k$-faces:
\[
f_k(Z^k(n,p,q)) \geq \tfrac{p}{2} (1-o(1)) \tbinom{n}{k+1}.
\]
Applying the second moment method it is not hard to see that the number of good $k$-faces, after choosing $a$, is at least $\frac{1}{2}(1-o(1)) \binom{n}{k+1}$ with probability tending to $1$. A Chernoff bound then tell us that a.a.s.\ $f_k(Y^k(n,p)) \geq \frac{p}{2} (1-o(1)) \binom{n}{k+1}$.
As $Y^k(n,p)$ is a subcomplex of $Z$, this yields the desired bound.
With a similar argument, also applying a Chernoff bound,  we get that a.a.s.
\[
|\delta a| \leq \frac{q}{2}(1-o(1)) \binom{n}{k+1}.
\]
As we have $\|[a]\|\geq 1/2-o(1)$ with the same probability as before, we see that a.a.s.\ 
\[
\varepsilon(Z) \leq \frac{\|\delta a\|}{\|a\|}=O\Big(\frac{q}{p}\Big) = o(1),
\]
if $q=o(p)$. In the extremal case $q = C\cdot \log(n)/n$ and $p=1$, we achieve $\varepsilon(Z)=O(\log(n)/n)$.

Furthermore, since $Z$ has $X^k(n,q)$ as a subcomplex, we know that the groups $\tilde{H}^{k-1}(Z,\Z_2) $ and $\tilde{H}_{k-1}(Z,\Z)$ are  a.a.s.\ trivial if $q\geq C\cdot\log n/n$ for $C$ sufficiently large (see \cite{HoffmanKahlePaquette-2013, LinialMeshulam:HomologicalConnectivityRandom2Complexes-2006, MeshulamWallach:HomologicalConnectivityRandomComplexes-2009}).

For the analysis of the spectrum of $\Lap_{k-1}^\up(Z)$, we can again consider the links of \mbox{$(k-2)$}-faces. For $H \in \binom{V}{k-1}$, the random graph $\lk(H,Z)$ is the union of $\lk(H,Y^k(n,p/2))$ and $\lk(H,Z^k(n,q))$. Hence, it has the distribution $G(n-k+1,r)$ with $r=p/2+q-pq/2$, the union of $G(n-k+1,p/2)$ and $G(n-k+1,q)$. As $r \geq p/2$,  we see that also for this construction, a.a.s., the nontrivial part of the spectrum of the normalized Laplacian $\Lap_{k-1}^\up(Z)$ lies in the interval $[1-\gamma_\Delta,1+\gamma_\Delta]$ with $\gamma_\Delta=O(1/\sqrt{rn})$.

\subsubsection*{Acknowledgements} The second author is grateful to Roy Meshulam for helpful discussions during which, in particular, he learned about Garland's results.
We would also like to thank Matt Kahle and the referees of this paper and of the extended abstract for helpful comments.

\bibliographystyle{abbrv}
\bibliography{EigenvaluesRandomComplexes}

\end{document}